\documentclass[letterpaper,reqno,12pt]{amsart} 
\usepackage{setspace}
\usepackage{graphicx} 
\usepackage{cite}
\usepackage{amssymb, enumerate, color}
\usepackage{hyperref}
\usepackage{enumitem}
\usepackage{graphicx}
\usepackage{subcaption}
\usepackage{here}
\usepackage[margin=1.2in]{geometry}
\usepackage{amsmath}
\usepackage{fancyhdr}
\usepackage{indentfirst}
\usepackage{tikz-cd}
\usepackage{tikz}
\usetikzlibrary{3d}
\usetikzlibrary{calc}
\usepackage{amsthm}
\usepackage[section]{placeins}
\usepackage{mathrsfs}
\usepackage[all]{xy}
\usepackage{setspace}
\usepackage{xcolor}
\usepackage[T1]{fontenc}
\usepackage[utf8]{inputenc}
\usepackage{cleveref}
\usepackage{musicography}
\hypersetup{colorlinks=true,colorlinks=true,linkcolor=blue,urlcolor=blue, citecolor=blue}

\DeclareMathOperator{\Bs}{Bs}
\DeclareMathOperator{\Supp}{Supp}

\DeclareMathOperator{\Proj}{Proj}
\DeclareMathOperator{\Spec}{Spec}

\DeclareMathOperator{\chara}{char}
\DeclareMathOperator{\sdeg}{sdeg}

\newcommand{\red}{\textcolor{red}}

\newcommand{\tp}{\mathrm{tp}}
\numberwithin{equation}{section} 

\theoremstyle{plain}

\newtheorem{theorem}{Theorem}[section]
\newtheorem{lemma}[theorem]{Lemma}
\newtheorem{proposition}[theorem]{Proposition}

\newtheorem{corollary}[theorem]{Corollary}

\newtheorem{definition}{Definition}[section]
\newtheorem{remark}{Remark}[section]
\newtheorem{example}{Example}[section]

\title{Boundedness of Stein degrees in positive characteristics}
\author{Xintong Jiang}
\address{Tsinghua University}
\email{xt-jiang21@mails.tsinghua.edu.cn}
\date{}

\begin{document}
\begin{abstract}
In this paper, we prove the Shokurov's conjecture in relative dimension $1$-case in sufficiently large characteristics. Depending on this result, we prove the boundedness of Stein degrees for horizontal boundaries in a log Calabi-Yau fibration in sufficiently large characteristics.
\end{abstract}
\maketitle
\tableofcontents
\section{Introduction}
We work over an algebraically closed field $k$ of characteristic $p>0$.  A log Calabi-Yau fibration is a projective contraction
\[
 f:(X,B)\longrightarrow Z
\]
such that $(X,B)$ is log canonical and $K_X+B\sim_\mathbb Q0/Z$.  For a proper dominant morphism $S\to Z$ of normal integral varieties, let
\[
 S\longrightarrow V\longrightarrow Z
\]
be its Stein factorization.  We write
\[
 \sdeg(S/Z)=\deg(V/Z)
\]
for the Stein degree of $S$ over $Z$.  If $S$ is not normal, we can apply this definition to its normalization $S^\nu$.  Equivalently, $\sdeg(S^\nu/Z)$ is the degree over $K(Z)$ of the algebraic closure of $K(Z)$ in $K(S)$.  Thus Stein degree isolates the finite part of a proper morphism and, in positive characteristic, retains both the separable and purely inseparable parts of the corresponding algebraic constant-field extension.

Stein degree was introduced by Birkar in the construction and boundedness of stable minimal models \cite{BirkarFano}.  There the Stein factorization records the number of connected components of the general fibres of a proper morphism; after normalization, the same invariant records their irreducible components.  In characteristic zero, for a log Calabi--Yau fibration $(X,B)\to Z$ of dimension $d$, every non-klt centre of $(X,B)$ has Stein degree over $Z$ bounded in terms of $d$ alone \cite{BirkarFano}.  This boundedness is an essential input in Birkar's moduli theory and is also of independent interest.

Birkar subsequently proposed the stronger statement $\mathcal H(d,t)$: if $S$ is a horizontal irreducible component of $B$ with $\mu_S B\ge t>0$, then $\sdeg(S^\nu/Z)$ should be bounded in terms of $d$ and $t$ only.  Birkar and Qu proved this conjecture in characteristic zero for generalized log Calabi--Yau fibrations, together with its vertical analogue $\mathcal V(d,t)$ for Fano type fibrations \cite{BirkarQu}.  Their proof follows the inductive route
\[
 \mathcal H(1,t)\Longrightarrow\mathcal V(1,t)
 \Longrightarrow\mathcal V(d,t)\Longrightarrow\mathcal H(d,t),
\]
using the minimal model program, bounded complements, bounded Fano varieties, and toroidal geometry.  For vertical divisors, the Fano type hypothesis is necessary: without it the strong Stein degree is unbounded \cite{BirkarQuNonFano}.

Our first result is the positive characteristic version of the Shokurov conjecture, which concerns bases of relative-dimension-one Fano type log Calabi-Yau fibrations. 
\begin{theorem}[\Cref{mult1}]
    Let $\epsilon>0$ be a real number and $R\subset [0,1]$ be a finite set of rational numbers, then there exists a prime number $p_0$ and a real number $\delta$ depending only on $\epsilon$ and $R$ such that suppose:
    \begin{enumerate}
        \item $(X,B)\to Z$ is an $\epsilon$-lc Fano-type fibration,
        \item $\dim X-\dim Z=1$,
        \item $B\in R$,
        \item $K_X+B\sim_\mathbb Q 0/Z$
        \item $\dim X\leq 3$ and
        \item $\chara k=p>p_0$,
    \end{enumerate} then the general fibers of $X/Z$ are $\mathbb P^1$ and there is a canonical bundle formula 
    \[K_X+B\sim_\mathbb Q f^*(K_Z+B_Z+M_Z)\]
    with the generalized pair $(Z,B_Z+M_Z)$ generalized $\delta$-lc. In particular, for any prime divisor $D$ on $Z$, the coefficients of $f^*D$ is bounded from above depending only on $\epsilon$ and $R$ after properly shrinking near the generic point of $D$.
\end{theorem}
Depending on the theorem above, we get our main theorem, which is the boundedness statement $\mathcal H(2,t)$ of three-dimensional horizontal Stein degrees. 

\begin{theorem}[$\mathcal H(2,t)$, \Cref{H2t}]
    Let $(X,B)/Z$ be a 3-dimensional log Calabi-Yau fibration with $\dim Z=1$, let $S$ be a horizontal component of $B$ with coefficient greater or equal than $t$, then there is a prime number $p_0$ depending only on $t$ such that suppose $\chara k>p_0$, then $\text{sdeg}(S^\nu/Z)$ is bounded from above depending only on $t$.
\end{theorem}
This is the positive-characteristic analogue of the threefold case of Birkar--Qu's theorem. Here we give the proof sketch of this theorem.

After a $\mathbb Q$-factorial dlt modification of $(X,B)$, we run the MMP for $K_X+(B-bS)$, where $b=\mu_SB$.  If its Mori fibre space has a surface base, then $S$ is ample and horizontal over that base and the relative-dimension-one bound $\mathcal H(1,t)$ together with the composition lemma completes the proof.  Otherwise $S$ is ample over the original curve and $X$ is of Fano type.  A $-(K_X+qS)$-MMP and the relative complement theorem produce a fixed-index complement $B^+\ge qS$. Now running the $K_X$-MMP gives a second Mori fibre space.

If the second target is a curve, its general fibre is a uniformly klt del Pezzo surface. After a log resolution unformly over $\Spec\mathbb Z$, the boundedness of $\epsilon$-lc del Pezzo surfaces over $\Spec\mathbb Z$ then bound the function field extension of the valuation $S$, whether or not $S$ survives the MMP.  If the second target is a surface, we perturb the pair so that $(X,B)$ is lc but not klt near $S$ and the problem divides into two branches depending on the whether the lc place near $S$ is horizontal over the second target. The horizontal lc-place branch uses an adjunction estimate to compare the degree of $S$ with that of a horizontal coefficient-one divisor.  In the vertical branch, the Shokurov's conjecture gives a uniform fibre-multiplicity threshold; selecting minimal multiplicities produces whole fibres, and a controlled redistribution of coefficient-one components produces an $\epsilon$-lc boundary $B^\flat$.   The final anticanonical MMP contracts to a model whose general fibre lies in a bounded family of del Pezzo surface pairs, where the bounded valuation statement bounds the function field extension of the original divisor $S$.

\section{Preliminaries}
In this paper, a variety $X$ is a reduced scheme separated of finite type over a field $k$. Moreover, $X$ is assumed to be quasi-projective over $k$, and $k$ is algebraically closed unless stated otherwise. For a scheme $X$, we assume every quasi-compact open subset of $X$ contains finitely many irreducible components unless stated otherwise, we define its reduction $X^\text{red}\to X$ to be the maximal reduced closed subscheme of $X$ and define its normalization $X^\nu\to X$ to be the normalization of $X$ under the morphism \[\coprod\limits_{\eta\to X}\Spec (k(\eta))\to X\] where $\eta$ varies among the generic points of all irreducible components of $X$, cf. \cite[\href{https://stacks.math.columbia.edu/tag/035E}{Tag 035E}]{stacks-project}. We denote $\eta$ as a field $k(\eta)$ and also as a generic point $\Spec(k(\eta))\to X$ if not causing confusion. It is well known that the normalization $X^\nu\to X$ factors through the reduction.

\subsection{Basics in birational geometry}
In this subsection,  we introduce the basic knowledge in birational geometry we shall use in this paper. We work over a field $k$. For a variety $X$, a resolution of singularity is a proper birational map $f:Y\to X$ from a smooth variety $Y$, which is an isomorphism over the regular locus of $X$, and for the singular locus $X_{sing}$, one has $f^{-1}X_{sing}$ is a divisor with simple normal crossings. 
\begin{theorem}
    For a 3-dimensional variety $X$, there is a resolution of singularity $f:Y\to X$ which is obtained by a sequence of blow-ups along smooth centers over $X_{sing}$.
\end{theorem}
\begin{proof}
    See \cite{Cutkosky2004ResolutionOS}, \cite{Cossart2008ResolutionOS} and \cite{Cossart2009RESOLUTIONOS}.
\end{proof} 

For $\mathbb F=\mathbb Q,\mathbb R$, $\mathbb F$-divisors are the $\mathbb F$-linear combination of divisors and $\mathbb F$-linear equivalences between divisors are generated by $\mathbb F$-linear combination of linear equivalences. A similar definition applies for $\mathbb F$-Cartier divisors. $(X,B)$ is called a sub-pair if $X$ is a normal variety and $B$ is a $\mathbb Q$-divisor such that $K_X+B$ is $\mathbb Q$-Cartier and $B\leq 1$ (in coefficients).  A sub-pair $(X,B)$ is called a pair if $B\geq 0$. Here, $B$ is called the boundary of the sub-pair.

For an $\mathbb F$-divisor $M$, we often denote \[H^0(M):=H^0(X,\mathcal O_X(\lfloor M\rfloor))=\{f\in k(X)\mid\text{div}(f)+M\geq 0\}.\] The linear system of $M$ is defined as  \[|M|:=\Proj(H^0(M))=\{N\sim M\mid N\geq0\}.\] The $\mathbb F$-linear system is defined as \[|M|_{\mathbb F}:=\{N\sim_{\mathbb F}\mid N\geq0\},\]in particular \[|M|_{\mathbb Q}=\bigcup\limits_{m\in\mathbb N} \frac{1}{m}|mM|.\]The base locus $Bs(|M|)$ denotes the maximal closed subset of $X$ contained in each $N\in|M|$, the stable base locus is defined as \[Bs(|M|_{\mathbb Q}):=\bigcap\limits_{m\in\mathbb N} Bs(|mM|).\] If $|M|\neq\emptyset$, $|M|$ will define a rational map \[\phi_M:X\dashrightarrow |M|^\lor\simeq \mathbb P^n,\] which is determined on $X\setminus Bs(|M|)$ by mapping $x$ to the hyperplane in $H^0(M)$ consisting of sections whose corresponding divisors $N\sim M$ pass through $x$. 

Let $X$ be a variety, a b-divisor $\mathbf M$ on $X$ is a configuration of divisors $\mathbf M_Y$ on each projective birational model $Y$ over $X$ such that if $f:Z\to Y$ is a morphism of birational models over $X$, then $f_*\mathbf M_Z=\mathbf M_Y$. A b-divisor is said to be represented by $Y$ if $\mathbf M_Y$ is $\mathbb R$-Cartier and if $\mathbf M_Z=f^*\mathbf M_Y$ holds for any projective birational morphism $f:Z\to Y$. Usually we will use $M:=\mathbf M_X$ to represent the b-divisor $\mathbf M$ for convenience.

We say a b-divisor $\mathbf M$ is b-nef if it is represented by some model $Y$ and $\mathbf M_Y$ is nef. Suppose $X$ is a $\mathbb Q$-factorial surface and $\mathbf M$ is a b-nef b-$\mathbb Q$-divisor on $X$, then $M$ is nef. Indeed, if $\mathbf M$ is represented by some $Y$ and $C$ is any curve on $X$, then \[M\cdot C=f_*\mathbf M_Y\cdot C=\mathbf M_Y\cdot f^*C\geq0,\] which implies $M$ is nef. In general, a b-nef b-divisor $\mathbf M$ is not nef on $X$.

A generalized pair is given as $(X',B'+M')/Z$ where $X'$ is a normal variety with a projective morphism $X'\to Z$, $B'\geq0$ a $\mathbb Q$-divisor (usually $B'\leq 1$) on $X'$ and a b-$\mathbb Q$-Cartier $b$-$\mathbb Q$-divisor $M'$ represented by some projective birational morphism $\phi:X\to X'$ and a $\mathbb Q$-Cartier $\mathbb Q$-divisor $M$ on $X$ such that $M$ is nef over $Z$ and $M'=\phi_*M$ and $K_{X'}+B'+M'$ is $\mathbb Q$-Cartier. Since $M'$ is defined birationally, one may assume that $X\to X'$ is a log resolution. $M$ is viewed as a b-divisor in generalized pairs.

Suppose $D$ is a prime divisor on $X$, for any $\mathbb Q$-divisor $A$ on $X$ we define $\mu_D(A)$ to be the coefficient of $D$ in $A$. For a prime divisor $D$ on a log resolution $W/X$ of the (resp. sub-)pair $(X,B)$, let $K_W+B_W$ be the pullback of $K_X+B$, the log discrepancy of $(X,B)$ is defined as \[a(D,X,B):=1-\mu_D(B_W).\] For a birational map $\phi:(X,B)\dashrightarrow(Y,B_Y)$, we say $\phi$ is crepant (resp. sub-crepant) if 
\[a(E,X,B)=\text{(resp.}\geq\text{) }a(E,Y,B_Y)\]for every divisor $E$ over $X$.

One say the (resp. sub-)pair $(X,B)$ is (resp.sub-)lc (resp. klt, plt, canonical, terminal, $\epsilon$-lc) if $a(D,X,B)\geq 0$ (resp. $>0$, $>0$ for exceptional $D$, $\geq 1$ for exceptional $D$, $>1$ for exceptional $D$, $\geq\epsilon$) for every $D$ over $X$. A non-klt place of a sub-pair $(X,B)$ is a prime divisor $D$ on birational models of $X$ such that $a(D,X,B)\leq0$. A non-klt center is the image on $X$ of a non-klt place. A (resp.sub-)pair is (resp.sub-)dlt if it is lc and log smooth near generic points of non-klt centers.

For a generalized pair $(X',B'+M')$ and a divisor $D$ over $X$, take a sufficiently high resolution $f:X\to X'$ defining $M'=f_*M$ and contains $D$, we define $K_X+B+M:=f^*(K_{X'}+B'+M')$, and one can similarly define generalized version of lc, klt, plt, $\epsilon$-lc by considering the generalized log discrepancy \[a(D,X',B'+M'):=1-\mu_D(B).\]If $M=0$, these notions of singularities will coincide with the classical version.

We use standard results of the minimal model program (MMP), MMP in char $k>5$ up to dimension 3 is already fully known in \cite{KYH20lcmmp}:
\begin{theorem}[\cite{KYH20lcmmp}, 1.1]\label{MMP}
    Let $(X,B)/Z$ be a 3-dimensional lc pair over $k$ of char $>5$, $X\to Z$ be a projective contraction, then there is a minimal model program$/Z$ on $K_X+B$ such that:
    \begin{enumerate}
        \item If $K_X+B$ is pseudo-effective$/Z$, then the MMP terminates with a log minimal model$/Z$.
        \item If $K_X+B$ is not pseudo-effective$/Z$,then the MMP terminates with a Mori fiber space$/Z$.
    \end{enumerate}
\end{theorem}

A normal variety is $\mathbb Q$-factorial if every divisor is $\mathbb Q$-Cartier. For a generalized pair $(X',B',M')$ with data $\phi:X\to X'$, a $\mathbb Q$-factorial generalized dlt modification is a $\mathbb Q$-factorial generalized dlt generalized pair $(X'',B''+M'')$ with a projective birational morphism $\psi:X''\to X'$ under a log resolution $X\to X''$ such that $B''$ and $M''$ are pushdowns of $B$ and $M$, where \[K_{X''}+B''+M''=\psi^*(K_{X'}+B'+M')\] and every exceptional prime divisor of $\psi$ that appears in $B''$ has coefficient 1. 

Such model exists for generalized lc pairs in dimension$\leq 3$. In fact, suppose $(X,B+M)$ is the log resolution of $(X',B'+M')$, we run $K_X+B+E+M$-MMP over $X'$, where 
\[E=\sum\limits_{E_i \in\text{Exc}(\phi)}a(E_i,X',B'+M')E_i,\]
we get a minimal model $\mathbb Q$-factorial generalized dlt model $\psi:X''\to X'$ with $E$ contracted by negativity lemma, and hence $\psi^*(K_{X'}+B'+M')=K_{X''}+B''+M''$.

If $(X',B'+M')$ is generalized klt, then the modification $(X'',B''+M'')$ is a $\mathbb Q$-factorial generalized klt model and $\psi$ is a small morphism (i.e. no divisor is contracted or extracted), this is called a small $\mathbb Q$-factorial modification.

Let $X$ be a normal projective variety of dimension $d$, and let $D$ be a $\mathbb Q$-divisor on $X$. The Kodaira dimension $\kappa(D)$ (resp. the numerical Kodaira dimension $\kappa_\sigma(D)$) is defined as $-\infty$ if $D$ is not effective (resp. pseudo-effective), and otherwise as the largest integer $r$ such that \[\limsup_{m\to\infty}\frac{h^0(\lfloor mD\rfloor)}{m^r}>0.\] resp. for some very ample divisor $A$\[\limsup_{m\to\infty}\frac{h^0(\lfloor mD\rfloor+A)}{m^r}>0.\] We define the volume \[\text{vol}(D):=\limsup_{m\to \infty}\frac{h^0(\lfloor mD\rfloor)}{m^d},\] and say $D$ is big if $vol(D)>0$. $|\lfloor mD\rfloor|$ will define a rational map \[\phi_m=\phi_{\lfloor mD\rfloor}:X-\Bs(\lfloor mD\rfloor)\to \mathbb P^\vee(H^0(X,mD)).\] The dimension of the image of $\phi_m$ will stabilize to $\kappa(D)$. The stabilized rational fibration $\phi:X\dashrightarrow \phi(X)$ is called the Iitaka fibration. Suppose $D$ is semi-ample, then $\phi$ is a morphism by definition.

A generalized log Calabi-Yau fibration of varieties is a projective surjective morphism $f:(X',B'+M')\to Z$ between varieties such that $(X',B'+M')$ is a generalized lc generalized pair with $K_{X'}+B'+M'\sim_\mathbb Q0/Z$. The following property of generalized log Calabi-Yau fibrations will be frequently used in this paper.

\begin{lemma}[\cite{sdeglogcalabiyau}, Lemma 2.17]\label{cycrep}
    Suppose $(X',B'+M')/Z$ is a generalized log Calabi-Yau fibration and $\phi:(X',B'+M')\dashrightarrow(Y',B_Y'+M'_Y)$ is a birational map that does not extract any divisor with $B_Y'=\phi_*B'$, then $\phi$ is a crepant map.
\end{lemma}

\subsection{Varieties of Fano type}
A contraction is a projective surjective morphism $f:X\to Y$ between schemes such that $f_*\mathcal O_X=\mathcal O_Y$, or equivalently the integral part of the Stein factorization of $f$ is trivial. For Stein factorization, see Section 3. Let $(X,B)$ be a pair with a contraction $X\to Z$, we say $(X,B)$ is log Fano (resp. weak log Fano) over $Z$ if $-(K_X+B)$ is ample (resp. nef and big) over $Z$. We say a pair $(X,\Delta)$ is of Fano type over $Z$ if there exists a boundary $B$ such that $(X,\Delta+B)$ is a klt weak log Fano pair over $Z$, or equivalently, if there exists a boundary $\Gamma$ such that $(X,\Delta+\Gamma)$ is klt, $\Gamma$ is big over $Z$, and $K_{X}+\Delta+\Gamma\sim_{\mathbb{Q}}0/Z$.

Suppose $f:X\to Y$ is a birational contraction and $(X,B)$ is of Fano type, then $(Y,B_Y=f_*B)$ is of Fano type as the pushforward of a big divisor is big. Suppose $(X,B)\dashrightarrow(Y,B_Y)$ is a sub-crepant birational map, then $(Y,B_Y)$ is of Fano type will imply that $(X,B)$ is of Fano type. Hence taking crepant resolutions, running MMP and taking $\mathbb Q$-factorial dlt models will keep the property of Fano type.

Let $X$ be a variety $\mathbb Q$-factorial of Fano type and suppose $(X,\Delta)$ is klt and $K_X+\Delta\sim_{\mathbb Q}0$. Let $D$ be a $\mathbb Q$-divisor on $X$, then for $\epsilon\ll 1$, we have \[\epsilon D\sim_{\mathbb Q} K_X+\Delta+\epsilon D\sim_{\mathbb Q}K_X+(1-n\epsilon)\Delta +n\epsilon \Delta+\epsilon D.\] Since $\Delta$ is big, we have $\Delta\sim_{\mathbb Q} B+A$ for some effective $\mathbb Q$-divisor $B$ and ample $\mathbb Q$-divisor $A$ on $X$, then we can always find some $n$ large enough such that there is some $ H\sim_{\mathbb Q} nA$ and $H+D>0$. So 
\[\epsilon D\sim_{\mathbb Q}K_X+(1-n\epsilon)\Delta +n\epsilon \Delta+\epsilon D\sim_{\mathbb Q}K_X+(1-n\epsilon)\Delta+ n\epsilon B+\epsilon (H+D).\]
Since $(X,\Delta)$ is klt, for $\epsilon\ll 1$, we always have $(X,(1-n\epsilon)\Delta+ n\epsilon B+\epsilon (H+D))$ is klt. Hence we can always run $\epsilon D$-MMP on $X$ which will terminate to get some good models by \Cref{MMP}. Moreover, in the outcome, if $D$ is a nef divisor on a threefold of Fano type of $\chara>5$, then $D$ is semi-ample by base-point free theorem \cite[1.2]{birkar2014existencemorifibrespaces}.

\subsection{Complements}

Now we introduce the terminology of complements, which is introduced by Shokurov in \cite{3foldflip}. Let $(X',B'+M')/Z$ be a generalized pair, set $T':=\lfloor B'\rfloor$ and $\Delta':=B'-T'$. An $n$-complement of $K_{X'}+B'+M'$ over $z\in Z$ is of the form $K_{X'}+B'^++M'$ such that over some neighbourhood of $z$, $(X,B'^++M')$ is generalized lc, $nM$ is b-Cartier, and \[n(K_{X'}+B'^++M')\sim 0,nB'^+\geq nT'+\lfloor (n+1)\Delta'\rfloor.\] Moreover if $(X',B'+M')$ is generalized klt, then we say $B^+$ is a klt $n$-complement. We say a complement $B^+$ is monotonic if $B^+\geq B$. We recall that monotonic relative $n$-complements exists in large characteristics for Fano type fibrations.

\begin{lemma}[\cite{complement1}, Theorem 1.7]\label{relacomp}
    Assume $(X,B)$ is a 3-dimensional projective lc pair, $f:X\to Z$ is a contraction with $3>\dim Z>0$, let $R\subset [0,1]$ be a finite set of rational numbers, then there is some natural number $n=n(R)\in\mathbb N$ a prime number $p_0=p_0(R)$ such that suppose:
    \begin{enumerate}
        \item $\chara k>p_0$
        \item $B\in\Phi(R)$,
        \item $X$ is of Fano type$/Z$,
        \item $-(K_X+B)$ is nef$/Z$.
    \end{enumerate}
    Then, for any $z\in Z$, there is an $n$-complement $K_X+B^+$ of $K_X+B$ over $z$ with $B^+\geq B$.
    
\end{lemma}

\subsection{Bounded families of schemes}
In this subsection, we introduce the notion of bounded families of schemes. A couple $(X,D)$ consists of a reduced scheme $X$ and a reduced divisor $D$ on $X$. Isomorphisms between couples are isomorphisms between the schemes such that the morphism is compatible and onto for boundaries. A set $\mathcal P$ of couples of projective varieties is birationally bounded (resp. bounded) over a scheme $S$ if there exist a couple $(\mathfrak X,\mathfrak D)$ and projective surjective morphisms $\mathfrak X\to \mathfrak T$ of reduced schemes of finite type over $S$. For each $(X,D)\in\mathcal P$ defined over a field $k$, there is a point 
\[\Spec(k)\to\mathfrak T,\] 
and a birational map (resp. isomorphism) $\phi:\mathfrak X_t\dashrightarrow X$ (resp. $\phi:\mathfrak X_t\to X$) such that the fiber $(\mathfrak X_t,\mathfrak D_t)$ of  $(\mathfrak X,\mathfrak D)\to \mathfrak T$ over $t$ is a couple of projective varieties over $t$ and $E\leq \mathfrak D_t$, where $E$ is the sum of the strict transform of $D$ and the reduced exceptional divisor of $\phi$. 

A set $\mathcal R$ of pairs of projective varieties $(X,B)$ is said to be log birationally bounded (resp. log bounded) over $S$ if the set of couples $(X,\text{Supp}B)$ is birationally bounded (resp. bounded) over $S$. If $B=0$, we simply say that the set of $X$ is birationally bounded (resp. bounded) over $S$. If $S=\Spec(k)$ is the base field, we simply say that $\mathcal R$ is log birationally bounded (resp. log bounded). A bounded family over $\Spec \mathbb Z$ usually implies that varieties in this family which are defined over a field of sufficiently large characteristic will admit  bounded good modifications. Explicitly speaking, we have the following propositions on geometric normality and bounded log resolution of singularities for a bounded family of varieties.

\begin{proposition}[{\cite[2.6]{complement1}}]\label{genorm}
    Suppose $\mathcal P$ is a bounded family over $\Spec\mathbb Z$ of projective varieties. Suppose a component $X_0\subset X\in\mathcal P$ is a normal variety defined over $k$, then there is $p_0=p_0(\mathcal P)$ such that if $\chara k>p_0$, $X_0$ is geometrically normal.
\end{proposition}

\begin{proposition}[Generic resolution of singularity over $\mathbb Z$]\label{genres}
    Let $(X,D)$ be a couple of finite type and dominant over $\Spec\mathbb Z$. Then there is an open subset $U\subset \Spec \mathbb Z$ and a projective birational morphism $\phi:X'\to X|_U:=X\times_{\Spec \mathbb Z} U$, such that $(X', D^\sim+\text{Exc}(\phi))\to U$ is log smooth, i.e. every irreducible stratum of $(X',D^\sim+\text{Exc}(\phi))$ is smooth over $U$, where $D^\sim$ denotes the strict transform of $D$.
\end{proposition}
\begin{proof}
    Let $(X^0,D^0)$ be the generic fiber of $(X,D)$ over $\Spec\mathbb Z$, which is a couple of finite type over $\Spec\mathbb Q$. By \cite[Theorem 1.4.1]{Resolution} there is a sequence of (weighted) blow ups of reduced centers 
    \[X'^0:=X^0_n\to X^0_{n-1}\to \cdots \to X^0_0:=X^0\]
    such that $(X'^0,E'^0:=E^0_n)$ is log smooth over $\mathbb Q$, where $E_k^0$ is the sum of strict transformations of $D^0$ and the reduced exceptional divisors of $X_k^0/X^0$. Denote $U_0=\Spec\mathbb Z$ and $(X_0,E_0)=(X,D)$. We construct $(X_i,E_i)\to U_i$ by induction on $i$. Suppose $f_i:X_i\to U_i$ is constructed, for the blow-up $X^0_{i+1}\to X_{i}^0$ at a reduced center $F_i^0$, let $F_i$ denote  the closure of $F_i^0\subset X_i$ over $U_i$, we define the open subset $U_{i+1} \subset U_i$ to be $U_i-f_i(F_i^{\text{ver}})$, where $F_i^\text{ver}$ denotes the vertical components of $F_i$. We define \[(X_{U_{i+1}},E_{U_{i+1}}, F_{U_{i+1}}):=(X_i,E_i,F_i)|_{U_{i+1}}.\] Blow up $X_{U_{i+1}}$ along the center $F_{U_{i+1}}$ and we get a scheme $X_{i+1}\to U_{i+1}$ whose generic fiber is $X_{i+1}^0$ by construction. We define $E_{i+1}$ to be the sum of the strict tranformations of $E_{U_{i+1}}$ and the reduced exceptional divisors of $X_{i+1}/X_{U_{i+1}}$, which coincide with $E^0_{i+1}$ on the generic fiber. Hence we finally get $\phi:X_n\to X|_{U_n}/ U_n$ with $(X_n, E_n)$ generically log smooth. By generic smoothness there is an open subset $U\subset U_n$ such that $(X_n, \text{Exc}(\phi))$ is log smooth over $U$. We denote $X':=X_n\times_{U_n}U$ and the assertion follows as desired.
\end{proof}

\begin{proposition}\label{genres1}
    Suppose $\mathcal P$ is a bounded family over $\Spec\mathbb Z$ of couples of projective integral normal varieties. Then there is a bounded family $\mathcal S$ over $\Spec \mathbb Z$ of log smooth couples and a prime number $p_0$ depending only on $\mathcal P$, such that if $(X,D)\in \mathcal P$ is defined over a field $k$ with $\chara k>p_0$, then $(X,D)$ admits a log resolution $(X^s,D^s)\in\mathcal S$.
\end{proposition}
\begin{proof}
    Let $(\mathfrak X,\mathfrak D)\to \mathfrak T/\Spec \mathbb Z$ be the bounded family of $\mathcal P$. Taking irreducible components of $\mathfrak T$ and ruling out the component of $\mathfrak T$ not dominating $\Spec \mathbb Z$, we may assume $\mathfrak T$ is irreducible and dominating $\Spec\mathbb Z$. Let $(\mathfrak X^s,\mathfrak D^s)\to \mathfrak T_U/U\subset\Spec\mathbb Z$ be the log smooth family constructed in \ref{genres}.   Then by generic smoothness \{\red{ref}\}, there is an open subset $\mathfrak V\subset \mathfrak T_U$ such that for any point $t\to \mathfrak V$, $(\mathfrak X^s_t,\mathfrak D^s_t)$ is log smooth. Hence for $t\in \mathfrak V$ corresponding to $(X,D)$, it admits a log resolution in the bounded family $(\mathfrak X^s_\mathfrak V,\mathfrak D^s_\mathfrak V)\to \mathfrak V/U\subset \Spec\mathbb Z$. Replacing $(\mathfrak X,\mathfrak D)\to \mathfrak T/\Spec \mathbb Z$ with $(\mathfrak X^s_{\mathfrak T_U\backslash\mathfrak V},\mathfrak D^s_{\mathfrak T_U\backslash\mathfrak V})\to  \mathfrak {T}_U\backslash\mathfrak V$ over $U\subset \Spec\mathbb Z$, the assertion follows by Noetherian induction.
\end{proof}

\subsection{Topological hyperplane sections}
Base-point free linear systems usually behave pathologically in positive characteristic. For example, let $L$ be a very ample line bundle on a variety $X$ over a field of characteristic $p>0$, and consider the base-point free sub-linear system $V=H^0(X,L)^p\subset H^0(X,L^p)$, we see every element in $|V|$ is not reduced. In this subsection, we introduce the technique of topological hyperplane sections in this subsection, which maximally avoids the pathological behaviors of base-point free linear systems.

\subsection*{Notations}
    We first introduce the settings we shall use in this paper. Let $K$ be a field (not necessarily algebraically closed), we denote $K^s$ (resp. $K^a$) to be the separable closure (resp. algebraic closure) of $K$. Suppose $X$ is a projective integral $K$-variety with $H^0(X,\mathcal O_X)=K$, $D$ is a Weil divisor on $X$, and $V\subset H^0(X,\mathcal O_X(D))$ is a $K$-linear subspace, the linear system associated to $V$ is defined as $|V|:=\mathbb P(V)$, which is a projective $K$-variety. We also define $|D|:=\mathbb P(H^0(X,\mathcal O_X(D)))$ to be the full linear system. We say  densely general elements (resp. general elements) of a linear system $\mathbb P(V)$ satisfy a property $\mathscr P$ if for a Zariski dense subset $U\subset \mathbb P(V)$ of $K$-points (resp. a non-empty open subset $U\subset \mathbb P(V)$), every effective Weil divisor $L\sim D$ corresponding to a $K$-point of $U$ satisfies $\mathscr P$.

    Let $f:X\to Z$ be a contraction between projective normal varieties over an algebraically closed field $k$, and let $i:Z\to \mathbb P_k^n$ be a closed immersion induced by a very ample line bundle $\mathscr L=i^*\mathcal O_{\mathbb P^n}(1)$. We denote $\Spec(k'):=\Spec(k(t_1,t_2,\cdots,t_n))$ to be the generic point of $\mathbb P(H^0(\mathbb P^n,\mathcal O_{\mathbb P^n}(1)))=(\mathbb P^{n})^\vee$, which is a purely transcendental extension of $k$. For any suffix $q=s,a,\prime,\cdots$ with a field extension $\Spec(k^q)\to\Spec(k)$ and any datum (schemes, divisors, line bundles and etc.) $\mathcal J/k$, we denote $\mathcal J^q/k^q$ to be the base change of the initial data along the field extension. 

\begin{definition}[Topological hyperplane sections]
    Let $k$ be a field and $f:X\to \mathbb P^n$ be a morphism of $k$-schemes, we define the topological extension of $k$ to be  
    \[k^\tp:=k(X_1,X_2,\cdots)^s=(\bigcup\limits_{n=1}^{\infty}k(X_1, X_2,\cdots,X_n))^s.\]
    For any couple $(X,D)$, we define its topologization to be $(X^\tp,D^\tp):=(X,D)\otimes_k k^\tp$. A topological hyperplane section of $f$ is a hyperplane section $(f^\tp)^*H^\tp\subset X^\tp$, where  $H^{\tp}\in ((\mathbb P^n)^\tp)^\vee$.
\end{definition}
The concept of topological extension is inspired by the following diagram of generic hyperplane section.
\[\begin{tikzcd}
	&& {\mathcal H_X} && {H'_X} && {H_X^\tp} \\
	X & {X\times(\mathbb P^n)^\vee} && {X'} && {X^\tp} \\
	&& {\mathcal H} && {H'} && {(H')^\tp} \\
	{\mathbb P^n} & {\mathbb P^n\times(\mathbb P^n)^\vee} && {(\mathbb P^n)'} && {(\mathbb P^n)^\tp} \\
	{\Spec (k)} & {(\mathbb P^n)^\vee} && {\Spec (k')} && {\Spec(k^\tp)}
	\arrow[from=1-3, to=2-2]
	\arrow[from=1-3, to=3-3]
	\arrow[from=1-5, to=1-3]
	\arrow[from=1-5, to=2-4]
	\arrow[from=1-5, to=3-5]
	\arrow[from=1-7, to=1-5]
	\arrow[from=1-7, to=2-6]
	\arrow[from=1-7, to=3-7]
	\arrow["p", from=2-1, to=4-1]
	\arrow[from=2-2, to=2-1]
	\arrow[from=2-2, to=4-2]
	\arrow[from=2-4, to=2-2]
	\arrow[from=2-4, to=4-4]
	\arrow[from=2-6, to=2-4]
	\arrow[from=2-6, to=4-6]
	\arrow["h"', from=3-3, to=4-2]
	\arrow[from=3-5, to=3-3]
	\arrow[from=3-5, to=4-4]
	\arrow[from=3-7, to=3-5]
	\arrow[from=3-7, to=4-6]
	\arrow[from=4-1, to=5-1]
	\arrow[from=4-2, to=4-1]
	\arrow[from=4-2, to=5-2]
	\arrow[from=4-4, to=4-2]
	\arrow[from=4-4, to=5-4]
	\arrow[from=4-6, to=4-4]
	\arrow[from=4-6, to=5-6]
	\arrow[from=5-2, to=5-1]
	\arrow[from=5-4, to=5-2]
	\arrow[from=5-6, to=5-4]
\end{tikzcd}\]
All squares in this diagram is Cartesian, $p$ is the composition of $f:X\to Z$ and $i:Z\to \mathbb P^n$, and $\mathcal H$ is the universal hyperplane section defined by 
\[\mathcal H:=(\sum\limits_{i=0}^ns_it_i=0)\subset \Proj(k[s_0,s_1,\cdots,s_n])\times\Proj (k[t_0,t_1,\cdots,t_n])\cong \mathbb P^n\times(\mathbb P^n)^\vee.\]
We see that $H_X^\tp$ is a topological hyperplane section of $X$, different choice of the isomorphism $\Proj (k[t_0,t_1,\cdots,t_n])\cong(\mathbb P^n)^\vee$ will give different topological sections. The terminology of such extension is named to be a "topologization" because for any scheme $X$ over a non-perfect field $K$, its separable closure $X^s$ will be homeomorphic to its algebraic closure $X^a$, however the geometry of $X^s$ and $X^a$ are quite different. We see that while the topological hyperplane section we have constructed has the same topological information with general geometric hyperplane sections, it also has the same algebro-geometric properties with the generic hyperplane section $H_X'$.
\begin{lemma}[Topologized Bertini]\label{topologicalbertini}
    Under the notations above, suppose $X\to Z$ is a contraction from a threefold to a surface over $k$, then we have:
    \begin{enumerate}
        \item $\dim_k |\mathscr L|=\dim_{k'}|\mathscr L'|=n$.
        \item For a general element $L$ of $|\mathscr L|$, $L$ is geometrically integral and geometrically normal curve. Moreover, assume general fibers of $X/Z$ are smooth curves, then $f^*L$ is integral and $f^*L\to L$ is a flat morphism. The conclusion is true when we replace $(f:X\to Z,L,\mathscr L)/k$ with $(f^q:X^q\to Z^q,L^q,\mathscr L^q)/k^q$ for arbitrary field extension $\Spec(k^q)\to\Spec(k)$.
        \item For a densely general element $L'$ of $|\mathscr L'|$, $f'^* L'$ is an integral normal $k'$-surface. Moreover, assume $(X,B)$ is a projective lc (resp. klt) pair, $(X',B'+f'^* L')$ is lc (resp. klt). The conclusion is true when we replace $(f':X'\to Z',L',\mathscr L')/k'$ with $(f^\tp:X^\tp\to Z^\tp,L^\tp,\mathscr L^\tp)/k^\tp$ for the topological extension $k\to k'\to k^\tp$.
    \end{enumerate}
\end{lemma}
\begin{proof}
    For $(1)$, we see that $|\mathscr L|=\mathbb P^n$ is of dimension $n$. Since $k\to k'$ is a purely transcendental extension, $H^0(Z,\mathscr L)\otimes _k k'$ is an $n+1$-dimensional $k'$-linear subspace of $H^0(Z', \mathscr L')$. Hence $\dim_{k'}|\mathscr L'|\geq n$. On the other hand, we see \[H^0(Z', \mathscr L')\otimes_{k'}\overline{k'}\subset H^0(Z_{\overline{k'}},\mathscr L_{\overline {k'}}).\] We see that $\mathscr L_{{k'}^a}$ induce a closed embedding $Z_{{k'}^a}\to \mathbb P^n_{{k'}^a}$ of hyperplane sections, which implies $\dim_{{k'}^a}H^0(Z_{{k'}^a},\mathscr L_{{k'}^a})=n+1$, hence $H^0(X_{k'},\mathscr L_{k'})=\mathbb A^{n+1}_{k'}$ as desired.

    For $(2)$, the statement about $\mathscr L$ is just the Bertini theorem for hyperplane sections, which holds for arbitrary infinite fields; cf. \cite[Theorem~3.4.14]{joinsandintersections}. As $X\to Z$ is a contraction, general fibers are irreducible,
    hence the geometrical irreducibility of $f^*L$ follows from \cite[6.10.3]{Bertini}. Now assume general fibers are smooth curves. Since $X/Z$ has geometrically reduced general fibers, we see for general sections in $|\mathscr L|$, the pull-back of such sections to $X$ are again reduced.
    So we have $f^*L$ is a geometrically integral surface over $k$ and $L$ is a smooth curve for general sections. By Bertini for Serre conditions (cf. \cite[Theorem~3.4.5]{joinsandintersections}),
    we have $f^*L$ is Cohen Macaulay for general $L$, so $f|_L:f^*L\to L$ is an equi-dimensional morphism from a Cohen Macaulay integral $k$-scheme to a smooth curve, which implies that $f|_L$ is flat by miracle flatness.
    For the conclusion under field extensions, since $f^*L\to L$ is defined over an algebraically closed field and the properties in the argument is stable under purely transcendental extensions, the properties will also hold under arbitrary field extensions.

    For $(3)$, the $k'$ part of the argument is simply \cite[Theorem 5.4]{Bertinibasechange}. We see $f^*L$ is geometrically irreducible by $(2)$, hence $(f^\tp)^*L^\tp$ is irreducible. The left properties in the argument are stable under base changing along separable extensions and purely transcendental extensions
    of the base field, hence the argument holds for $(f^\tp)^*L^\tp$.
\end{proof}
\begin{remark}
    Keep the notations in \Cref{topologicalbertini}, in general, better regularity than reducedness for general $f^*L$ needs more assumptions on $X/Z$. For example, one can prove that suppose $X\to Z$ is a contraction from a threefold to a surface with the following assumptions:
    \begin{enumerate}
        \item general fibers of $X/Z$ are irreducible smooth curves,
        \item for every codimension-$1$ point $\alpha\in Z$ and every generic point $\beta$ of $f^{-1}\alpha$, the residue field extensions  $\kappa(\beta)/\kappa(\alpha)$ is separable,
    \end{enumerate} 
    then for general elements $L\in |\mathscr L|$, $f^*L$ is an integral normal $k$-surface.
\end{remark}

\section{Stein factorization and Stein degree}
For a universally closed quasi-separated locally of finite type surjective morphism $f:X\to Y$ between schemes, the Stein factorization of $f$ is the unique factorization, cf. \cite[\href{https://stacks.math.columbia.edu/tag/03GX}{Tag 03GX}]{stacks-project}:
\[\begin{tikzcd}
	X & T \\
	& Y
	\arrow["g", from=1-1, to=1-2]
	\arrow["f"', from=1-1, to=2-2]
	\arrow["h", from=1-2, to=2-2]
\end{tikzcd}\]
Here $T=\Spec_Y(f_*\mathcal O_X)$, $g_*\mathcal O_X=\mathcal O_T$ and $h$ is integral, moreover $g$ is universally closed quasi-separated locally of finite type and surjective. Suppose $f$ is proper and $Y$ is locally Noetherian, e.g. $X$ and $Y$ are varieties over a field $k$, we have $g$ is proper with geometrically connected fibers and $h$ is finite.

\begin{example}
    Let $Y=\mathbb A^1_k:=\Spec(k[t])$ for a field $k$ of characteristic $p>0$, and let $F:Y\to Y$ denote the absolute Frobenius morphism of $Y$. Denote by 
    \[Y\to Y'=\Spec (k[t^p])\simeq \mathbb A^1_k\]
    the geometric Frobenius morphism of $Y$, which is a $k$-variety morphism. Let $X=\mathbb A^1\times \mathbb P^1\to Y$ be the projection to the first factor and consider the composite morphism $f:X\to Y\to Y'$. We see $Y\to Y'$ is the finite part of the Stein factorization of $f$. We see that though $f$ has geometrically irreducible fibers, the finite part of the Stein factorization of $f$ is not trivial.
\end{example}
\subsection{Stein degree}
Suppose $f:X\to Y$ is a universally closed locally of finite type quasi-separated morphism between locally Noetherian schemes and $Y$ is integral with the generic point $\eta$, let $h:T\to Y$ be the Stein factorization of $f$, we define the Stein degree $\sdeg(X/Y)$ of $X\to Y$ to be the $\eta$-dimension of $\Gamma(T_\eta,\mathcal O_{T_\eta})=\Gamma(X_\eta,\mathcal O_{X_\eta})$.
It is evident that $\sdeg(X/Y)\leq \sdeg (X^\nu/Y)$.

From now on we assume $X$ is normal and $f$ is proper. Let $\{X_\alpha\}_{\alpha\in A}$ be the irreducible components of $X$. By definition, we have \[\sdeg(X/Y)=\sum\limits_{\alpha\in A}\sdeg(X_\alpha/Y).\] So we assume $X$ is normal integral for convenience.

\begin{proposition}\label{sdeg_normal}
    Let $f:X\to Y$ be a proper surjective morphism between integral locally Noetherian schemes with $X$ normal, and let $h\colon T\to Y$ be the finite part of the Stein factorization of $f$. Then $T$ is an integral normal scheme, and the Stein degree is $\sdeg(X/Y)=\deg(K(T)/K(Y))=\deg(T/Y)$. Moreover, $K(T)$ is the algebraic closure of $K(Y)$ in $K(X)$.
\end{proposition}
\begin{proof}
    By  \cite[\href{https://stacks.math.columbia.edu/tag/035L}{Tag 035L}]{stacks-project},
    $T$ is an integral normal scheme. By definition, \[\sdeg(X/Y)=\dim_{k(\eta)}\Gamma(T_\eta,\mathcal O_{T_\eta})=\dim_\eta\Gamma(\eta_T,\mathcal O_{\eta_T})=\dim_\eta K(T)=\deg(K(T)/K(Y)).\]
    For $K(T)$, this is a generic problem hence we may assume $X$ and $Y$ are affine schemes and $\mathcal O_T$ is the integral closure of $\mathcal O_Y$ in $\mathcal O_X$. As $\mathcal O_X$ is normal, suppose some element $\alpha\in K(X)$ is algebraic over $K(Y)$, there is some element $\beta\in \mathcal O_Y$ such that $\beta\alpha$ is integral over $\mathcal O_Y$ and hence integral over $\mathcal O_X$, hence $\beta\alpha\in\mathcal O_X$ and by definition $\beta\alpha\in\mathcal O_T$, which implies that $\alpha\in K(Y)\mathcal O_T\subset K(T)$ and hence $K(T)$ is the algebraic closure of $K(Y)$ in $K(X)$.
\end{proof}
\subsection{Separable Stein degree and purely inseparable Stein degree}
By \Cref{sdeg_normal}, the Stein degree between integral locally Noetherian schemes  $X/Y$ with $X$ normal is equal to the Stein degree of the function field extension $\Spec(K(X))/\Spec(K(Y))$. So we only need to describe Stein degree for essentially of finite type field extensions. Let $K(Z)/K(Y)$ be the finite part of the Stein factorization, where $Z\to Y$ is the finite part of the Stein factorization of $X\to Y$. Let $K(Y)\to K(Y)^a$ be an algebraic closure of $K(Y)$ and denote by $k$ the separable closure of $K(Y)$ in $K(X)$, which is a subfield of $K(Z)$ by \Cref{sdeg_normal}. Then we have:
\[\begin{aligned}
    \sdeg(X/Y)&=\deg(K(Z)/K(Y))=\deg(K(Z)\otimes_{K(Y)} K(Y)^a/K(Y)^a)\\&=\deg(K(Z)\otimes_k(k\otimes _{K(Y)}K(Y)^a)/K(Y)^a)
\end{aligned}\]
We define the separable Stein degree of $X/Y$ to be 
\[d_s:=\deg(k/K(Y))=\deg(k\otimes_{K(Y)}K(Y)^a/K(Y)^a).\] Since $k$ is separable over $K(Y)$, we have \[k\otimes_{K(Y)}K(Y)^a\simeq (K(Y)^a)^{\oplus d_s}\] as $K(Y)^a$-algebras.
Hence $d_s=\mathfrak n(\Spec (k\otimes_{K(Y)}K(Y)^a))$, where $\mathfrak n(X)$ denotes the number of irreducible components of $X$ for a scheme $X$. We also denote $\mathfrak n(A):=\mathfrak n(\Spec(A))$ for a ring $A$. More generally, for a morphism $f:X\to Y$ of schemes, we define $\mathfrak n(X/Y)$ to be the number of geometric irreducible components of a general fiber of $X/Y$, which is well-defined by \cite[\href{https://stacks.math.columbia.edu/tag/0BUI}{Tag 0BUI}]{stacks-project}.
Moreover, in our case that $X$ and $Y$ are integral normal schemes with Stein factorization $X\to Z\to Y$, we have 
\[\begin{aligned}
    \mathfrak n(X/Y)&=\mathfrak n(X_{\bar \eta})=\mathfrak n(K(X)\otimes_{K(Y)}K(Y)^a)\\
    &=\mathfrak n(K(X)\otimes_k(K(Y)^a)^{\oplus d_s})=d_s\mathfrak n(K(X)\otimes_{K(Z)}(K(Z)\otimes_kk^a))\label{2.3.1}
\end{aligned}\]
Here $k^a=K(Y)^a$ is an algebraic closure of $k$. We see that $K(Z)/k$ is purely inseparable hence radicial, that is $K(Z)/k$ is a universal homeomorphism and hence \[\mathfrak n((K(Z)\otimes_kk^a))=\mathfrak n(k^a)=1.\] 

Now we are going to show that the base change along $K(Z)\to K(X)$ preserve the number of irreducible components, or equivalently $\mathfrak n(K(X)\otimes_{K(Z)}(K(Z)\otimes_kk^a))=1$. Denote 
\[K(Z)^\text{perf}:=\lim\limits_{e\to\infty} F^e_* K(Z)\]
to be the colimit perfection of $K(Z)$, which is a perfect field purely inseparable over $K(Z)$. Since $K(Z)$ is algebraically closed in $K(X)$, $K(Z)^\text{perf}\otimes_{K(Z)} K(X)$ is a purely inseparable extension of $K(X)$, which is radicial. We claim that $K(Z)^\text{perf}$ is algebraically closed in $K(Z)^\text{perf}\otimes_{K(Z)} K(X)$. In fact, suppose 
\[\alpha\in K(Z)^\text{perf}\otimes_{K(Z)} K(X)-K(Z)^\text{perf}\]
is an element algebraic over $K(Z)^\text{perf}$, denote \[f(x)=\sum\limits_{i=0}^n a_ix^i\] to be the monic minimal polynomial of $\alpha$ in $K(Z)^\text{perf}$, then there is an $e\in\mathbb N$ such that $a_i^{p^e}\in K(Z)$ for all $i=0,1,\cdots, n$ and $\alpha^{p^e}\in K(X)$. We have that \[f^e(x):=\sum\limits_{i=0}^n a_i^{p^e}x^i\in K(Z)[x]\] is a zero polynomial of $\alpha^{p^e}\in K(X)$, which implies that $\alpha^{p^e}\in K(Z)$ as $K(Z)$ is algebraically closed in $K(X)$, and hence $\alpha\in K(Z)^\text{perf}$. 

Hence after base changing to $K(Z)^\text{perf}$ and replacing $K(X)/K(Z)$ with the field extension $K(Z)^\text{perf}\otimes_{K(Z)} K(X)/K(Z)^\text{perf}$, we may assume $K(Z)$ is a perfect field. Let $K'$ be the colimit perfection of $K(X)$, $K(Z)$ is algebraically closed in $K'$ since $K(Z)$ is perfect and $K'/K(X)$ is purely inseparable. Rename $K(Z)$ with $k'$ and replace $K(X)$ with $K'$, we end with the field extensions $K'/k'/k$, where $k'/k$ is purely inseparable, $K'$ and $k'$ are perfect field and $k'$ is algebraically closed in $K'$. It suffices to prove that \[\mathfrak n(K'\otimes_{k'}(k'\otimes_kk^a))=1.\]

$R:=k'\otimes_k k^a$ is a ring with a unique prime ideal $\mathfrak p$ such that $R/\mathfrak p\simeq k^a$.  Since $K'$ is $k'$-flat, $\mathfrak p\otimes _{k'}K'\to R\otimes _{k'}K'$ is an injection and $\mathfrak p\otimes _{k'}K'$ is an ideal of $R\otimes _{k'}K'$. Now what we are going to prove is that $\mathfrak n(R\otimes_{k'}K')=1$. We see that $\mathfrak p$ is the nilpotent radical of $R$, hence every element in $\mathfrak p\otimes_{k'}K'$ is nilpotent. We have \[\mathfrak n(R\otimes_{k'}K')=\mathfrak n(R\otimes_{k'}K'/\mathfrak p\otimes_{k'}K')=\mathfrak n(k^a\otimes_{k'}K'),\] thus it suffices to prove that $k^a\otimes_{k'}K'$ is irreducible. We only need to check this for finite extensions of $k'$ instead of base changing to $k^a$ by passing to colimits. Let $l/k'$ be a finite extension, which is finite separable, hence a single extension, say $l=k'[\theta]$ with $f\in k'[t]$ the monic minimal polynomial of $\theta$. We claim that $f$ is also the minimal polynomial of $\theta'$ in $K'$. In fact factorizing $f$ in $K'^a$ and $k'^a$ will give the same unique factorization, hence if the minimal polynomial $h$ of $\theta$ in $K'$ has degree not equal to $f$, then $h|f$ and all coefficients are algebraic over $k'$, hence in $k'$ and $h$ will serve as a zero-equation of $\theta$, which contradicts the minimality of $f$. This implies that \[l\otimes_{k'} K'=K'(\theta)\] is also a field separable over $K'$ as desired. Hence $\mathfrak n(K(X)\otimes_{K(Z)}(K(Z)\otimes_kk^a))=1$ and in particular we have \[\mathfrak n(X/Y)=d_s\mathfrak n(K(X)\otimes_{K(Z)}(K(Z)\otimes_kk^a))=d_s.\] 

So the Stein degree of $X/Y$ is equal to $\mathfrak n(X/Y)d_i$, where $d_i$ is the purely inseparable degree of $K(Z)/K(Y)$. We denote $\mathfrak i(X/Y)$ to be $d_i$ and we have $\sdeg(X/Y)=\mathfrak n(X/Y)\mathfrak i(X/Y)$ and we call $\mathfrak n(X/Y)$ and $\mathfrak i(X/Y)$ the separable Stein degree and the purely inseparable Stein degree of the morphism $f:X\to Y$ between normal integral schemes instead of $d_s$ and $d_i$ respectively. Suppose $X$ is not normal generally, we have 
\[\sdeg(X/Y)\leq\sdeg(X^\nu/Y)=\mathfrak n(X/Y)\mathfrak i(X/Y).\]
More generally, for a dominant map $f:X\to Y$ with $X$ proper over $\text{Im}(f)$, where $\text{Im}(f)$ is an open subscheme of $Y$ whose base space is the image of $f$, we define the Stein degrees $\sdeg(X/Y),\sdeg(X^\nu/Y),\mathfrak n(X/Y),\mathfrak i(X/Y)$ to be the Stein degrees of $X/\text{Im}(f)$. Moreover, suppose $X,Y$ are $K$-schemes with $K$ separably closed, then $\mathfrak n(X/Y)$ equals to the number of irreducible components of a general closed fiber by definition. In particular, Stein degrees between schemes defined over an algebraically closed field of characteristic $0$ is much simpler than that in general case.
\begin{corollary}
    Suppose $X,Y$ are normal integral schemes with characteristic $0$ function fields and $f:X\to Y$ is a dominant map proper over its image, then $\sdeg(X/Y)=\mathfrak n(X/Y)$.
\end{corollary}
The following proposition shows that the Stein degrees are well controlled under base changes and compositions.
\begin{proposition}\label{composition}
    Let $X\xrightarrow{f} Y\xrightarrow{g} Z$ be projective surjective morphisms between integral schemes with $[K(Z):K(Z)^p]=p^d$ for some $d\in \mathbb N$,  we have:
    \begin{enumerate}
        \item $\mathfrak n(X/Y)=\mathfrak n(X\times_Y Y'/Y')$ for any dominant morphism $Y'\to Y$, 
        \item $\mathfrak i(X/Y)=\mathfrak i(X\times_Y Y'/Y')$  for any generically separable dominant morphism $Y'\to Y$,
        \item $\mathfrak n(X/Z)\leq \mathfrak n(X/Y)\mathfrak n(Y/Z)$,
        \item $\mathfrak i(X/Z)\leq \mathfrak i(X/Y)^d\mathfrak i(Y/Z)$,
        \item $\sdeg(X^\nu/Z)\leq \sdeg(X^\nu/Y)^{\max\{d,1\}}\sdeg(Y^\nu/Z)$.
    \end{enumerate}
\end{proposition}
\begin{proof}
    For $(1)$, suppose $Y'\to Y$ is a dominant morphism, then any general geometric point $y\to Y'$ is also a general geometric point of $Y$ and the assertion  follows.

    For $(2)$, this is a purely field-theoretical problem hence we may assume $X$ is integral. Denote $K(Y)$ to be $k$, $K(X)$ to be $K$, the separable closure and the algebraic closure of $k$ in $K$ to be $k^s$ and $k^a$ respectively. Suppose $k\to l$ is a separably generated field extension, then $l$ is a separable algebraic extension of $k(\{t_i\}_{i\in I})$ for some index set $I$. It's clear that purely transcendental extensions preserve the inseparable degree. Hence it suffices to show the argument for separable algebraic extensions $k\to l$. By passing to limit, we may assume $k\to l$ is finite, hence a single extension, say $l=k(\alpha)$. Hence $k^s\otimes_k l$ is the finite direct sum of several $k^s(\alpha)$'s. Replacing $k\to l$ with $k^s\to k^s(\alpha)$ we only need to treat with $k$ separably closed in $K$. We see that suppose $k$ is separably closed in $K$ and $l$ is a separable extension of $k$, then $L:=l\otimes_k K$ is a field and $l$ is separably closed in $L$. Moreover, as $k^a$ is purely inseparable over $k$, we see \[k^a\otimes_k l=lk^a=:f\] is a field finite separable over $k^a$ and purely inseparable over $l$ by considering the elements in the fields. Hence \[L=K\otimes_k l=K\otimes_{k^a}f.\] By considering the elements in the fields, again we have $f$ is algebraically closed in $L$. Hence $\mathfrak i(K\otimes_k l/l)=[f:l]=[k^a:k]=\mathfrak i(X/Y)$.

    For $(3)$, this is the same as in \cite[Lemma 2.9]{sdeglogcalabiyau}. Here we give a field theoretical proof. Let $K\to L\to M$ be the function field extensions. If one of such $\mathfrak n(L/K)$ or $\mathfrak n(M/L)$ is infinity then we are done, so we may assume $\mathfrak n(L/K),\mathfrak n(M/L)<\infty$. Replace $M$ with the separable closure of $L$ in $M$ we may assume $M/L$ is a finite separable extension. Let $K^s$ denote the separable closure of $K$ in $L$ and $K'$ denote the separable closure of $K$ in $M$. It suffices to show $[K':K^s]\leq [M:L]$. In fact suppose $\beta\in K'$ is separable over $K^s$, then consider the minimal polynomial $f$ of $\beta$ over $K^s$, since $K^s(\beta)$ is separable over $K^s$, we see that $f$ is also the minimal polynomial of $\beta$ in $K^a$ as we have \[[K^a(\beta):K^a]=[K^s(\beta):K^a\cap K^s(\beta)]=[K^s(\beta):K^s]=\deg(f).\]  Hence suppose $f_1$ is the monic minimal polynomial of $\beta$ in $L$ which divides $f$, then the coefficients of $f_1$ are algebraic over $K$ and hence $f_1$ is the minimal polynomial of $\beta$ in $K^a$, which is just $f$. So we have \[[M:L]\geq [L(\beta):L]= [K^s(\beta):K^s].\] Passing to the colimit of all finite sub-extensions of $K'/K^s$, which are  such single extensions, we have $[M:L]\geq [K':K^s]$ as desired.

    For $(4)$, we may assume $X$ is integral. Therefore we only need to consider the function field extensions $K\to L\to M$ and prove that \[\mathfrak i(M/K)\leq \mathfrak i(L/K)^d\cdot\mathfrak i(M/L).\] If one of such $\mathfrak i(L/K)$ or $\mathfrak i(M/L)$ is infinity then we are done, so we may assume $\mathfrak i(L/K),\mathfrak i(M/L)<\infty$. Replace $L$ with its separable closure $L'$ in $M$, let $K'$ to be the algebraic closure of $K$ in $L'$, we show that $K'/K^a$ is separable. Otherwise there is an inseparable element $x\in K'\subset L'$ over $K^a$, consider its minimal polynomial $f$ over $K^a$ and its minimal polynomial $g$ over $L$, then $g$ is separable and $f$ is not separable. However $g$ divides $f$ and every coefficient of $g$ is algebraic over $K^a$ and falls in $L$, hence in $K^a$, which is a contradiction as $g\neq f$. Hence we have $\mathfrak i(L/K)=\mathfrak i(L'/K)$. Thus we may assume $L$ is separably closed in $M$ and $K$ is separably closed in $L$. Let \[L=L_0\subset L_1\subset L_2\subset\cdots\subset L_n=L^a\subset M\] be a chain of purely inseparable extensions of degree $p$ with $L^a$ algebracally closed in $M$. Moreover, we have $\mathfrak i(M/L)=p^n$. We prove the argument by induction on $n$. If $n=0$, then the assertion follows as the algebraic closure of $K$ in $L$ is the same as that in $M$. Suppose the $(n-1)$-case is proved, consider the $n$-case. Now twe are going to prove that $\mathfrak i(L_n/K)\leq p^d\mathfrak i(L_{n-1}/K)$. Replace $L$ with $L_{n-1}$, we only need to show the statement for $n=1$. In fact, let $K^b$ denote the algebraic closure of $K$ in $L_1$, we see $K^b/K^a$ is purely inseparable. Since $L_1/L$ is purely inseparable of degree $p$, we have $L_1=L(\alpha)$ for some $\alpha^p=a\in L$. Then suppose $\beta\in L'$ is algebraic over $K$, then we see $\beta^p\in L$ is algebraic over $K$, hence $\beta^p\in M$. Since $K^a$ is algebraic over $K$, we see \[[K^b:K^a]\leq [(K^a)^\frac{1}{p}:K^a]\leq [K:K^p]=p^d.\] Hence we have\[\mathfrak i(L_n/K)\leq p^d\mathfrak i(L_{n-1}/K)\leq\cdots\leq p^{dn}\mathfrak i(L/K)=\mathfrak i(M/L)^d\mathfrak i(L/K)\] by induction hypothesis and the assertion follows.

    For $(5)$, this directly follows by $(3)$ and $(4)$.
\end{proof}

\section{Boundedness of S-degree of divisors on log Calabi-Yau fibrations}

Before going to the proof of the boundedness of the Stein degrees, we first state a characteristic-free version of the Shokurov's Conjecture in relative dimension $1$-case. We work over an algebraically closed field $k$ with characteristic $p$.
\begin{theorem}\label{mult1}
    Let $\epsilon>0$ be a real number and $R\subset [0,1]$ be a finite set of rational numbers, then there exists a prime number $p_0$ and a real number $\delta$ depending only on $\epsilon$ and $R$ such that suppose:
    \begin{enumerate}
        \item $(X,B)\to Z$ is an $\epsilon$-lc Fano-type fibration,
        \item $\dim X-\dim Z=1$,
        \item $B\in R$,
        \item $K_X+B\sim_\mathbb Q 0/Z$
        \item $\dim X\leq 3$ and
        \item $\chara k=p>p_0$,
    \end{enumerate} then the general fibers of $X/Z$ are $\mathbb P^1$ and there is a canonical bundle formula 
    \[K_X+B\sim_\mathbb Q f^*(K_Z+B_Z+M_Z)\]
    with the generalized pair $(Z,B_Z+M_Z)$ generalized $\delta$-lc. In particular, for any prime divisor $D$ on $Z$, the coefficients of $f^*D$ is bounded from above depending only on $\epsilon$ and $R$ after properly shrinking near the generic point of $D$.
\end{theorem}
\begin{proof}
    The following proof is similar to that in \cite{SingbaseFanofib}. By \cite[Theorem 3.4]{complement1}, we see the general fibers are $\mathbb P^1$ and the canonical bundle formula holds for $p>\frac{2}{\min(R)}$, hence it suffices to show that $(Z,B_Z+M_Z)$ is generalized $\delta$-lc for some $\delta$. Replace $X$ and $Z$ with their $\mathbb Q$-factorializations we may assume $X$ and $Z$ are $\mathbb Q$-factorial. We are going to bound the coefficients of $B_Z$. The boundedness of coefficients of the b-divisor $\mathcal B_Z$ can be obtained by taking higher models of $X/Z$.
    
    Let $D\subset Z$ be a prime divisor contained in $\Supp(B_Z)$, shrink $Z$ near the generic point of $D$ we may assume $D$ is a Cartier divisor and each component of $B$ is either horizontal over $Z$ or mapped onto $D$. Pick a topological hyperplane section $L^\tp$ such that $(X^\tp, B^\tp+(f^\tp)^*L^\tp)$ is plt with second minimal log discrepancy $\geq \epsilon$, denote $X_L^\tp:=(f^\tp)^*L^\tp$ and we have adjunction formula 
    \[K_{X_L^\tp}+B_L^\tp:=(K_{X^\tp}+B^\tp+X_L^\tp)|_{X_L^\tp}.\]
    Moreover, we have $(X_L^\tp,B_L^\tp)$ is $\epsilon$-lc, and $D_L^\tp$ is sum of $k^\tp$-points, we only need to consider one of them. In the following paragraphs, denote $(X,B)\to Z$ to be $(X_L^\tp,B_L^\tp)\to L^\tp$, which is a lc surface pair projective over a smooth curve over $k^\tp$ with general fibers geometrically integral and geometrically normal.
    
    We put $b=1-\frac{\epsilon}{2}$ and let $\phi:W\to X$ be a log resolution of $(X, B+f^*D)$. Let $\{M_i\}$ be the set of components of $\phi^*f^*D$, let $\{M_j'\}$ be the set of prime exceptional divisors of $\phi$ which do not belong to $\{M_i\}$, and let $\{M_k''\}$ be the set of components of the strict transform of $B$ which do not belong to $\{M_i\}$. We have 
    \[\phi^*B=\sum\limits_i a_i M_i+\sum\limits_j b_j M''_j+\sum\limits b_k M'_k,\]
    where $b_j\leq 1-\epsilon$ by definition. Define
    \[\begin{aligned}
        \Delta_W&:=\sum\limits_i M_i+\sum\limits_jbM'_j+\sum\limits_k b_kM''_k,\\
        \Gamma_W&:=\sum\limits_i M_i+\sum\limits_jbM'_j+\sum\limits_k bM''_k,\\
        B_W&:=\phi^*(K_X+B)-K_W.
    \end{aligned}\]
    We have $(W,\Gamma_W)$ and $(W,\Delta_W)$ are dlt. Any component of $B_W$ with positive coefficient is either exceptional over $X$ or a component of the strict transform of $B$, which is a component of $\Gamma_W$ automatically. So we have $\Supp(\phi^*B)\subset\Supp(\Gamma_W-B_W)$ and $K_W+\Gamma_W$ is big over $Z$ as $X$ is of Fano type over $Z$. Let \[\Gamma:=\phi_*\Gamma_W,\quad-E_W:=-(K_W+\Gamma_W)+\phi^*(K_X+\Gamma),\] we have $E_W$ is exceptional over $X$. Let $g:W\to Z$ be the composite morphism and let $(G,\Gamma_G)$ be a general fiber of $g$, and let $(F,\Gamma_F)$ be a general fiber of $f$ corresponding to $G$. We have that $F$ is isomorphic to $G$ and \[\begin{aligned}K_F+\Gamma_F&=(K_X+\Gamma)|_F=\phi^*(K_X+\Gamma)|_G\\&=(K_W+\Gamma_W-E_W)|_G=K_G+\Gamma_G-E_G=K_G+\Gamma_G\end{aligned}\] as $E_W$ is exceptional over $X$ and hence vertical over $Z$. Since $\Supp \Gamma=\Supp B$ over the generic point of $Z$, we have $\Supp\Gamma_F=\Supp B_F$ has bounded degree, and for each component $S\subset \Supp B$, $S|_F$ is reduced for $p>\frac{2}{\min (R)}$ as $B\subset R$. Hence $\Gamma_F=\Gamma|_F\leq \Supp B_F$ and we get
    \[K_F+\Gamma_F\leq K_F+\Supp B_F\sim_\mathbb Q \Supp B_F-B_F\leq \Supp B_F\]
    has bounded degree from above, say bounded by $v$. Run $(K_W+\Gamma_W)$-MMP over $Z$ we end with a good log minimal model $h:(Y,\Gamma_Y)\to Z$. For a general fiber $(H,\Gamma_H)$ of $h$, it is isomorphic to $(G,\Gamma_G)$ as the fibers are curves and $K_W+\Gamma_W$ is effective over the generic point of $Z$. Let $\phi:Y\to Y'$ be the birational contraction defined by $K_Y+\Gamma_Y$ over $Z$ and let $T$ be a component of $\Supp\lfloor \Gamma_Y\rfloor=\Supp h^*D$ which is not contracted over $Y'$. By adjunction formula, cf. \cite[Proposition 2.8]{DasWeakBAB} we have \[K_{T^\nu}+\Gamma_{T^\nu}:=(K_Y+\Gamma_Y)|_{T^\nu}\] is nef and big (hence ample as $T^\nu$ is a normal curve) with $(T^\nu,\Gamma_{T^\nu})$ lc and $\Gamma_{T^\nu}\in \Phi $ for some DCC set $\Phi$ of rational numbers as $\Gamma_Y\in\{b,1\}$. Hence $\deg_{k^\tp}(K_{T^\nu}+\Gamma_{T^\nu})\geq \theta$ for some $\theta$ depending only on $\epsilon$ by direct calculations. For general $k^\tp$-point $J$ with fiber $(H,\Gamma_H)$ other than $D$, we have
    \[\begin{aligned}
         v\geq\deg_{k^\tp}(K_H+\Gamma_H)&=\deg_{k^\tp}(K_Y+\Gamma_Y)|_{h^*J}=(K_Y+\Gamma_Y)\cdot h^*J\\&=(K_Y+\Gamma_Y)\cdot h^*D=\sum_im_i(K_Y+\Gamma_Y)\cdot T_i\\&=\sum_im_i\deg_{k^\tp}(K_{T_i^\nu}+\Gamma_{T_i^\nu}),
    \end{aligned}\]
    where the last equality holds by pulling back the algebraic classes to the normalization of $T^\nu$. Suppose some $K_{T_i^\nu}+\Gamma_{T_i^\nu}$ is big for some $i$, then $v\geq m_i\theta$ and hence $m_i\leq \frac{v}{\theta}$ as desired.

    Now we are going to get rid of $T_i$'s such that $K_{T_i^\nu}+\Gamma_{T_i^\nu}$ is not big. We run $(K_Y+\Delta_Y)$-MMP$/Y'$ with scaling of $P_Y=\Gamma_Y-\Delta_Y$ we end with a good minimal model, replace $Y$ with that model and we have $K_Y+\Delta_Y\sim_\mathbb Q-P_Y/Y'$ is semi-ample over $Y'$. As $K_{Y'}+\Gamma_{Y'}$ is ample over $Z$, we have $K_Y+\Gamma_Y-rP_Y$ is semi-ample over $Z$ for sufficiently small $r$. In fact, we have 
    \[K_Y+\Gamma_Y-rP_Y=\phi^*(K_{Y'}+\Gamma_{Y'})+r(-P_Y)\sim_\mathbb Q \phi^*(K_{Y'}+\Gamma_{Y'}+rD)+rH\]
    is semi-ample over $Z$, where $-P_Y\sim_\mathbb Q\phi^*D+H$ for some semi-ample $\mathbb Q$-divisor $H$. So $(K_Y+\Gamma_Y-rP_Y)$ induces a contraction $Y\to Y''$ over $Z$ such that $Y''\to Y'$ is a morphism for sufficiently small $r$. We see that every component of $\lfloor \Gamma_Y\rfloor$ that is contracted over $Y'$ is contracted over $Y''$. Indeed, suppose $T$ is a component of $\lfloor \Gamma_Y\rfloor$ that is contracted over $Y'$, then $(K_Y+\Gamma_Y)|_T$ has non-positive degree and so for $K_Y+\Gamma_Y-rP_Y$ as $P_Y$ is effective, so $T$ is contracted over $Y''$. 
    
    Now we have components of $h''^*D$ has coefficients bounded by $\frac{v}{\theta}$ where $h'':Y''\to Z$ is the induced morphism as each component of $h''^*D$ is a component of $h^*D$ which is not contracted over $Y'$. Define $B_{Y''}$ to be the strict transform of $B_W$ on $Y''$. Since $B_W+\epsilon\lfloor \Delta_W\rfloor\leq \Delta_W$, we have $B_{Y''}+\epsilon\lfloor\Delta_{Y''}\rfloor\leq \Delta_{Y''}$. As $\Supp h''^*D=\Supp\lfloor\Delta_{Y''}\rfloor$ and coefficient of $h''^*D$ is bounded from above by $\frac{v}{\theta}$, hence we have
    \[K_{Y''}+B_{Y''}+\frac{\epsilon\theta}{v}h''^*D\leq K_{Y''}+\Delta_{Y''}.\]
    Hence, $(Y'',B_{Y''}+\frac{\epsilon\theta}{v}h''^*D)$ is sub-lc, which implies that $(X,B+\frac{\epsilon\theta}{v}f^*D)$ is lc by \Cref{cycrep} as $(X,B)$ is log Calabi-Yau over $Z$, hence the coefficient of $D$ in $B_Z$ is less than or equal to $1-\frac{\epsilon\theta}{v}$, which implies that the multiplicities in $f^*D$ is bounded by $\frac{v}{\epsilon\theta}$ as desired.
\end{proof}

\begin{theorem}[$\mathcal H(2,t)$]\label{H2t}
    Let $(X,B)/Z$ be a 3-dimensional log Calabi-Yau fibration with $\dim Z=1$, let $S$ be a horizontal component of $B$ with coefficient greater or equal than $t$, then there is a prime number $p_0$ depending only on $t$ such that suppose $\chara k>p_0$, then $\text{sdeg}(S^\nu/Z)$ is bounded from above depending only on $t$.
\end{theorem}
\begin{proof}
    We may assume $(X,B)$ is dlt $\mathbb Q$-factorial by taking a such modification. We can run a $-S$-MMP over $Z$ and we end with a Mori fiber space $X'\to Z'/Z$ and $S$ is not contracted. Suppose $\dim Z'>\dim Z$, then $S$ is ample over $Z'$ and hence horizontal over $Z'$, we have \[\text{sdeg}(S^\nu/Z)=\text{sdeg}(S^{\prime\nu}/Z)\leq \sdeg(S^{\prime\nu}/Z')\cdot\sdeg(Z'/Z)=\sdeg(S^{\prime\nu}/Z')\]
    is bounded from above by induction hypothesis and \Cref{composition}. Thus we can assume $Z'\to Z$ is an isomorphism of curves and $S$ is ample over $Z$, therefore $(X,B)$ is of Fano type over $Z$. We run a $-(K_X+tS)\sim (B-tS)$-MMP over $Z$ we end with a minimal model $(X',B')/Z$. Since $S$ is an ample divisor, we see its pushdown $S'$ on $X'$ is also a prime divisor and horizontal over $Z$. Replacing $X$ with $X'$, we can assume $-(K_X+tS)$ is semi-ample over $Z$. By \cite[Theorem 1.7]{complement1}, there is a prime number $p_0$ depending only on $t$ such that for every $\chara k=p>p_0$, there is a monotonic $n$-complement $(X,B^+)$ of $(X,tS)$ over the generic point of $Z$. Shrink $Z$ and replace $(X,B)$ with the dlt $\mathbb Q$-factorial model of  $(X,B^+)$ we may assume $B\in \mathbf T_n:=\frac{\mathbb Z}{n}\cap [0,1]$ and $X$ is klt. By generic global ACC, we see $(X,0)$ is $\epsilon$-lc near the generic fiber over $Z$ for some $\epsilon$ depending only on $t$. Now shrink $Z$ and run $K_X$-MMP over $Z$, we end with a Mori fiber space $f:X'\to Z'/Z$ and $(X',0)$ is $\epsilon$-lc.

    Suppose $\dim Z'=1$, that is $Z'\to Z$ is an isomorphism of curves, then $X_\eta$ is an $\epsilon$-lc del Pezzo surface over $\eta$ of Picard number $1$. As $-nB_{\eta}\sim nK_{X_\eta}$, by \cite{bernasconi2024boundinggeometricallyintegraldel} such $(X'_\eta,B'_\eta)$ falls in a log bounded family $/\mathbb Z$, hence geometrically normal by \Cref{genorm}. Now we are going to keep track with the divisor $S$ over $X'$. Suppose $S$ is not contracted during the $K_X$-MMP, then we see $S_\eta$ is a component of $B_\eta$, which is bounded over $\mathbb Z$. Thus the Stein degree of $S_\eta$ over $\eta$ is bounded from above depending only on $\epsilon$, hence on $t$. So we may assume $S$ is contracted during this minimal model program. By \Cref{genres1}, we see that for $p$ large enough depending only on $t$, there is a log smooth log resolution $(Y'_\eta,B_{Y'_\eta})\to (X'_\eta,B'_\eta)$ such that the strata of $\Supp(B_{Y'_\eta})$ belongs to a log smooth bounded family over $\mathbb Z$. Since $(X'_\eta,B'_\eta)$ is the outcome of $K_{X_\eta}\sim_\mathbb Q-B_\eta$-MMP and $S_\eta$ is a divisor contracted during this MMP with $a_\eta(S,X,B)\leq 1-t<1$, we have the center $C$ of $S_\eta$ on $Y_\eta$ belongs to a strata of $\Supp(B_{Y_\eta})$, whence $C$ has bounded Stein degree over $\eta$. As $(Y,B_Y)$ is log smooth, $S_\eta\to C$ is a contraction, which shows that $\sdeg(S_\eta/\eta)=\sdeg(C/\eta)$ is bounded from above depending only on $t$.
    
    Suppose $\dim Z'=2$, let $\phi:X''\to X'$ be the contraction which only extracts the divisorial valuation at $S$ and $X''$ is $\mathbb Q$-factorial. Set $B'$ to be the pushdown of $B$ to $X'$ and denote $K_{X''}+B'':=\phi^*(K_{X'}+B')$ and let $S''$ be the strict transform of $S$ on $X''$. Suppose $S''$ is horizontal over $Z'$, then we have
    \[\sdeg(S^\nu/Z)=\sdeg (S''^\nu/Z)\leq\sdeg(S''^\nu/Z')\sdeg(Z'/Z)=\sdeg(S''^\nu/Z')\]
    is bounded from above depending only on $t$ by applying induction hypothesis on $(X'',B'')/Z'$ and \Cref{composition}. Suppose $S''$ is not horizontal over $Z'$, then  its schematic image $T'\subset Z'$ is a prime divisor horizontal over $Z$ as $S''$ is horizontal over $Z$.  Run a $-S''$-MMP over $Z'$ we end with a good minimal model $h''':(X''',B''')\to Z'$, we have  $S'''$ is the whole fiber over its image $T'\subset Z'$. Replace $(X,B)$ with $(X''',B''')$ we may assume $f:(X,B)\to Z$ admits such a factorization $h:(X,B)\to Z'/Z$.

    Since the statement only depends on the generic informations over $Z$, we may shrink $Z$ freely in the following proofs. Pick a sufficiently small $\epsilon>0$ and suppose $h:(X,B)\to Z'$ is not $\epsilon$-lc near the generic point of $Z$, then there is a horizontal lc place $D$ over $X$ by ACC of lct for threefolds, extract $D$ if necessary. Suppose all lc center of $(X,B)$ does not intersect $S$, then the centers are vertical over $Z'$ as $S$ is the whole fiber over a divisor on $Z'$. Suppose $D$ is such a lc place which does not intersect $S$, we extract $D$ if necessary. We run $-D$-MMP over $Z'$ and get $(X',B')$, we see $S$ is not contracted during this MMP as $D$ does not intersect $S$. So we may assume $-D$ is nef over $Z'$, hence is the whole fiber over the prime divisor $D'=h(D)$. Repeat this process we can get a model $(X'',B'')/Z'$ where all horizontal$/Z$ lc centers of $(X'',B'')$ is contained in some whole fiber $D$ over $D'\subset Z'$ and $D$ itself is a lc place. Hence $B''=T+\Delta$ where $T=\lfloor B\rfloor$ contains all such lc places. We claim that for $p$ large enough $T$ is an irreducible divisor. By \cite[Theorem 3.4]{complement1} we have a canonical bundle formula for the Fano type fibration $(X'',B'')/Z'$ for large $p$:
    \[K_{X''}+B''\sim_\mathbb Q h^*(K_{Z'}+B_{Z'}+M_{Z'})\sim_\mathbb Q0/Z,\] restrict the formula to the generic point $\eta$ of $Z$, we get $Z'_\eta$ is a smooth curve of genus $0$ and $h(D)_\eta$ is a Cartier divisor on $Z'_\eta$, which is a component of $B_{Z',\eta}$. Suppose there are two such $D$'s, then $B_{Z',\eta}$ contains two components of coefficient $1$ and one component $S_{Z',\eta}$ of coefficient at least $t$, hence \[0=\deg (K_{Z'_\eta}+B_{Z',\eta}+M_{Z',\eta})\geq -2+2+t=t>0\] since $M_{Z'}$ is pseudo-effective by the canonical bundle formula, which is a contradiction. Hence there is at most one such $D$ and $h(D)_\eta$ is a rational point on $Z'_\eta$ when $p$ is large enough. 
    
    Extract lc places of $(X'',B'')$ near $D$ which are horizontal over $Z$ and we get a dlt model $h^\circ:(X^\circ,B^\circ)\to Z'$, let $D_1,D_2\cdots D_m$ be the lc places over $h(D)$ and $E_1,E_2\cdots E_l$ be other components over $h(D)$. Consider \[(h^\circ)^*(h(D))=\sum\limits_{i=1}^m a_i D_i+\sum\limits_{j=1}^l b_j E_j\] where $a_i$'s are positive integers as $h(D)_\eta$ is a Cartier divisor on $Z'_\eta$, hence we may assume $a_1$ is the minimal one among all possible $a_i$'s. We run $-D_1$-MMP over $Z'$ and end with the model $h^\sharp:(X^\sharp,B^\sharp)\to Z'$ where $D_1$ plays the role of $D$ in $(X,B)$, which is the whole fiber over $h(D)$. We see that \[(X^\sharp, B^\sharp-D_1)=(X^\sharp,B^\sharp-\frac{1}{a_1}(h^\sharp)^*(h(D))),\] 
    which is $\epsilon=\frac{1}{n}$-lc as all lc places of $(X,B)$ touches $D$ and all of them has log discrepancy $\geq 1$ in $(X^\sharp, B^\sharp)$ and all other places in $(X,B)$ has log discrepancy $\geq \frac{1}{n}$. So we have \[h^\sharp:(X^\sharp,B^\sharp-D_1)\to Z'\] is an $\epsilon$-lc Fano type log Calabi-Yau fibration with boundary coefficients in $\mathbf T_n$. We replace $(X,B, D)$ with $(X^\sharp, B^\sharp,D_1)$ by \Cref{cycrep}. By \Cref{mult1} applied on $(X,B-D)\to Z'$, the multiplicities of fibers in $X$ over divisors in $Z'$ is bounded from above depending only on $\epsilon$. Say $h^*(h(D))=lD$ for some integer $l$ bounded from above, we define \[B^\flat :=B-\frac{1}{2}D+\frac{1}{2l}h^*L\] for some general integral divisor $L$ such that $L_\eta\sim h(D)_\eta$ is a Cartier divisor on $Z'_\eta$, we have $(X_\eta,B_\eta^\flat)\to Z_\eta$ is a log Calabi-Yau fibration with $\epsilon$-lc singularity for $\epsilon<\frac{1}{2}$. Replace $B$ with $B^\flat$ and we may assume $(X,B)$ is $\epsilon$-lc.
    
    Now we run $-K_X\sim_\mathbb Q B$-MMP$/Z$ and we end with a model $X\to X'/Z$ defined by the semi-ample fibration, $\dim X'=3$ as $X$ is of Fano type over $Z$. We have $(X',B')$ is $\epsilon$-lc and $-K_{X'}$ is ample over $Z$, hence the generic fiber $X'_\eta$ is an $\epsilon$-lc del Pezzo surface. Moreover, We see $(X'_{\eta}, B'_\eta)$ belongs to a log bounded family $(\mathfrak X,\mathfrak B)\to \mathfrak T/\mathbb Z$ for $p$ large enough as $2nlB\sim -2nlK_X$ over $Z$ by our contruction. By \Cref{genres1}, we can construct a log smooth family $\mathcal P:=(\mathfrak X^s,\mathfrak B^s)\to \mathfrak T^s$ such that for $p>p_0$ large enough, $(X',B')$ admits a log resolution $(X^s, B^s)$ which is parametrized by $\mathcal P/\mathfrak T$. Hence $S$ contracts to some strata of $B^s$ which implies that $\sdeg(S/Z)$ is bounded from above as desired.
    
    So we assume there exists some lc place $D$ horizontal$/Z$ whose center intersects $S$, we assume $D\neq S$ first. Then consider the adjunction formula \[(K_X+B)|_{D^\nu}=K_{D^\nu}+B_{D^\nu},\] there is a component $S_{D^\nu}$ in $\Supp (B_{D^\nu})$ with coefficient $\geq t$ in $B_{D^\nu}$. Suppose $S_{D^\nu}$ is not horizontal over $Z$, we simply replace $Z$ with $Z-f(S_{D^\nu})$. So we assume $S_{D^\nu}$ is horizontal over $Z$. Suppose $D^\nu\to U\to Z$ is the Stein factorization of $D^\nu\to Z$ and $S^\nu\to R\to Z$ is the Stein factorization of $S^\nu\to Z$. Consider the following commutative diagram:
    \[\begin{tikzcd}
	{S^\nu_{D^\nu}} & {D^\nu} & U \\
	{S^\nu} & R & Z
	\arrow[from=1-1, to=1-2]
	\arrow[from=1-1, to=2-1]
	\arrow[from=1-2, to=1-3]
	\arrow[from=1-3, to=2-3]
	\arrow[from=2-1, to=2-2]
	\arrow[from=2-2, to=2-3]
\end{tikzcd}\]
Then we see that $S_{D^\nu}$ dominates $R$ and $U$. Suppose that $\sdeg(D^\nu/Z)=\deg(U/Z)$ is bounded from above, then by \Cref{composition}, we have 
\[\sdeg(S^\nu_{D^\nu}/Z)\leq \sdeg(S^\nu_{D^\nu}/U)\cdot\deg(U/Z)\]
is bounded from above by induction hypothesis applied to the contraction $(D^\nu,B_{D^\nu})\to U$. Hence $\deg(T/Z)$ is bounded from above by $\sdeg(S^\nu_{D^\nu}/Z)$ as desired. If $D$ is horizontal over $Z'$, then this is again done by $\mathcal H(1,t)$ and \Cref{composition}. So we may assume $h(D)$ is a prime divisor on $Z'$ which intersect $h(S)$ and we have\[h(S_{D^\nu})\subset h(S)\cap h(D),\] which is also a closed subset of $h(S)$ horizontal over $Z$. Hence we have $h(S)=h(D)$ as reduced divisors, in particular, it suffices to prove the case for $S=D$ is itself the lc center by running $-D$-MMP over $Z'$. 

Now replace $S$ with the lc place $D$ and do the same replacement for other lc centers as the discussions above, we may assume the following assumptions by shrinking $Z$ properly:
\begin{enumerate}
    \item $(X,B)$ is lc and $n(K_X+B)\sim0/Z$,
    \item all lc centers of $(X,B)$ are vertical over $Z'$ but horizontal over $Z$,
    \item each of the lc centers is contained in a unique lc place $D_i\subset X$ which is the whole fiber over its image in $Z'$,
    \item none of the images of $D_i$'s in $Z'$ intersects and
    \item one of such $D_i$'s is $S$.
\end{enumerate}
This case can be similarly settled by using the techniques above. We still state the full arguments here. Under these assumptions, we can construct a higher model $\psi:(X',B')\to (X,B)$ satisfying the following properties by taking a dlt modification and extracting the lc places properly:
\begin{enumerate}
    \item $(X',B')$ is a $\mathbb Q$-factorial dlt pair crepant over $(X,B)$,
    \item for each $D_i$, denote $\mathfrak G_i$ to be the set of all lc places over $(X,B)$ whose center is contained in $D_i$, there exists some $D_i^0:=D_i,D_i^1,D_i^2,\cdots, D_i^{m_i}\in \mathfrak G_i$ for some natural number $m_i$ such that 
    \[\sum\limits_{j=0}^{m_i}D_i^j=\lfloor B\rfloor\text{ over }h(D_i)\subset Z',\]
    \item $D_i^{m_i}$ has the minimal possible multiplicity $l_i$ in $(h')^*h(D_i)$ among all lc places in $\mathfrak G_i$.
\end{enumerate}
We run $-D_i^{m_i}$-MMP on $(X',B')$ over $Z'$ successively for all $i$, as $h(D_i)$ does not intersect, we have $D_i^{m_i}$ are not contracted in the ending model $(X^\sharp, B^\sharp)$, which is crepant to $(X',B')$ by \Cref{cycrep}. The image of $S$ in $X^\sharp$ is contained in some component $S^\sharp=D_{i_0}^{m_{i_0}}$ for some $i$. Thus its suffices to prove $\sdeg(S^{\sharp\nu}/Z)$ is bounded from above depending only on $t$. We see that \[h^\sharp:(X^\sharp, B^\sharp-\sum\limits_i D_i)\to Z'\]
is an $\epsilon=\frac{1}{n}$-lc Fano type log Calabi-Yau fibration with boundary coefficients in $\mathbf T_n$. By \Cref{mult1} all $l_i$'s are bounded depending only on $\epsilon$ for large enough $p$, hence only on $t$. We define
\[B^\flat:=B^\sharp-\frac{1}{2}\sum\limits_i D_i^{m_i}+\frac{1}{2}\sum\frac{1}{l_i}(h^\sharp)^* L_i\]
for some general integral divisor $L_i\subset Z'$ which is linearly equivalent to $h(D_i)$ over the generic point $\eta$ of $Z$. We see that the following properties holds by shrinking $Z$ properly:
\begin{enumerate}
    \item $(X^\sharp,B^\flat)$ is $\epsilon$-lc for all $\epsilon<\frac{1}{2}$,
    \item $(X^\sharp,B^\flat)$ is log Calabi-Yau over $Z$,
    \item $S^\sharp$ is a component in $B^\flat$ with coefficient $\frac{1}{2}$ and
    \item$2n\prod\limits_il_iB^\flat\sim-2n\prod\limits_il_iK_{X^\sharp}/Z$.
\end{enumerate}
Finally, we run $-K_{X^\sharp}\sim_\mathbb Q B^\flat$-MMP over $Z$, replace it with the ending model and we have a birational morphism $(X^\sharp,B^\flat)\to (X^\natural,B^\natural)/Z$ defined by the semi-ample fibration. We have $(X_\eta^\natural,B_\eta^\natural)$ is in a log bounded family, hence $\sdeg(S^\sharp/Z)$ is also bounded from above depending only on $t$ by the same reason as above, as desired.
\end{proof}

Besides the horizontal component, one can bound the separable Stein degree for vertical components whose image is also a divisor by using the following lemma.

\begin{lemma}\label{lccompn}
    Let $K$ be a separably closed field of characteristic $p>5$ and $X/K$ be a geometrically integral normal projective surface with $H^0(X,\mathcal O_X)=K$. Suppose $(X,\Delta)$ is a lc pair with $\Delta\in \Phi(R)$ for some finite set of rational numbers $R$ and $x\in X$ is a closed point, then the number of irreducible components of $\Delta$ passing through $x$ with coefficients $\geq t$ is bounded from above depending only on the real number $t>0$.
\end{lemma}
\begin{proof}
    Suppose $X$ is lc but not klt at $x$, then we are done since no component of $\Delta$ would pass through $x$. Now we assume that $X$ is klt at $x$. If $(X,\Delta)$ is lc but not klt at $x$, we see that there is a plt blow-up $f:(X',\Delta')\to (X,\Delta)$ such that $f$ extracts exactly one lc place $C=\lfloor \Delta'\rfloor$ and $K_{X'}+\Delta'=f^*(K_X+\Delta)$. Applying the adjunction formula, cf. \cite[Proposition 2.8]{DasWeakBAB}, we see $0\sim_\mathbb Q (K_{X'}+\Delta')|_{C}\sim_\mathbb QK_C+\Delta_C$ and $C$ is a normal curve of arithmetic genus $0$ over $k(x)$, hence either a conic over $k(x)$ or a projective line. Suppose any component $D'$ of $\Delta'$ with coefficient $1-a$ rather than $C$ contributes at least a component $D_C\in \Delta_C$ with coefficient $1-\frac{a}{n}\geq 1-a$ for some natural number $n$. Thus suppose there are $r$ components of $\Delta$ which has coefficient $\geq t$, then \[2=-\deg K_C=\deg \Delta_C\geq (r-1)t.\] As a result, $r\leq \frac{2}{t}+1$. Suppose $(X,\Delta)$ is klt, let $E$ be a component of $\Delta$, then we replace $\Delta$ by $\Delta+\text{lct}_x (E,X,\Delta) E$ we see $(X,\Delta)$ is lc but not klt and the number of components whose coefficient $\geq t$ does not decrease.
\end{proof}
\begin{theorem}[$\mathcal V(1,t)$]\label{V1t}
    Let $t>0$ be a real number and let $f:(X,B)\to Z$ be a 3-dimensional Fano type log Calabi-Yau fibration of relative dimension $1$. Let $S$ be a vertical irreducible component of $B$, there exists a prime number $p_0=p_0(t)$ and a natural number $n=n(t)$ such that suppose
    \begin{enumerate}
    \item $\chara k>p_0$,
        \item the coefficient of $S$ in $B$ is greater or equal than $t$,
        \item the image $T$ of $S$ on $Z$ is a prime divisor.
    \end{enumerate}
    Then $\mathfrak n(S/T)\leq n$.
\end{theorem}
\begin{proof}
    We may assume $t$ is a rational number. Suppose $p>5$, then we see general fibers of $X/Z$ are normal curves as $X/Z$ is of Fano type. Take a dlt $\mathbb Q$-factorial model we may assume $(X,B)$ is dlt $\mathbb Q$-factorial. Shrink $Z$ near the generic point of $T$ we may assume $Z$ is regular. Run an $-S$-MMP  on $X$ over $Z$ and shrink $Z$ near the generic point of $T$ we may assume $-S$ is $f$-nef as $S$ is vertical and maps to a prime divisor on $Z$. Since $-S$ is $f$-nef, so $f^{-1}(T)$ is irreducible with the reduced component $S$. Run $-(K_X+tS)$-MMP over $Z$, we may assume $-(K_X+tS)$ is nef over $Z$. As $S$ is the whole fiber over $T$, it is not contracted under this MMP. By \cite[Theorem 1.7]{complement1}, there exists a prime number $p_0(t)$ and a natural number $n(t)$  such that if $p>p_0$, there is a monotonic $n$-complement $(X,B^+)$ of $K_X+tS$ over the generic fiber of $T$. As $X$ is of Fano type over $Z$, there is a horizontal component $H$ of $B^+$ such that the coefficient of $H$ in $B^+$ is greater or equal than $a:=\frac{1}{n}$. 

    By $\mathcal H(1,t)$, the $\text{sdeg}(H^\nu/Z)$ is bounded from above by $N$ and $K(Z)\to K(H)$ is separable for sufficiently large $p$. Run $-H$-MMP over $Z$ we end with a model $(X',B')$ and $H$ is not contracted during this MMP. Since $S$ is the whole fiber over $T$, $S$ is not contracted either during this MMP. So we end with a Mori fiber space $X'\to Z'/Z$ with the image of $S$ a prime divisor $T'$ on $Z'$ and $H$ is ample over $Z'$. Replace $Z$ with $Z'$ and $(X,B)$ with $(X',B')$ we may assume $X\to Z$ is an $-H$-Mori fiber space. Hence for a general reduced fiber $S_z$, $H$ intersects every component of $S_z$. Pick a topological hyperplane section $L^\tp$ on $Z$, set $X_{L^\tp}$ to be the pull back of $L^\tp$ on $X$. By \Cref{topologicalbertini}, we see $(X^\tp,X_{L^\tp}+B^\tp)$ is a Fano type log Calabi-Yau fibration over $Z^\tp$. By adjunction formula, we see $(X_{L^\tp},B_{L^\tp})\to L^\tp$ is also a log Calabi-Yau fibration. We see that $H_{L^\tp}$ is irreducible by Bertini theorem and the coefficient of $H_{L^\tp}$ in $B_{L^\tp}$ is greater or equal than $a$. Pick a general closed point $z\in {L^\tp}\cap T^\tp$, we see $H_{L^\tp}$ intersects every component of $S_z$. We see that there are at most $N$ intersection points on $H$ as $\sdeg(H^\nu/Z)\leq N$. By \Cref{lccompn}, there are at most $\frac{2}{t}+1$ components of $S_z$ passing though one such intersection point, so there are at most $N(\frac{2}{t}+1)$ components in $S_z$, which is bounded depending only on $t$.
\end{proof}
\bibliographystyle{alpha}
\bibliography{cite}

\newcommand{\etalchar}[1]{$^{#1}$}
\begin{thebibliography}{AbH{\etalchar{+}}25}

\bibitem[AbH{\etalchar{+}}25]{Resolution}
Dan {Abramovich}, Andr{\'e} {belotto da Silva}, Ming {Hao Quek}, Michael {Temkin}, and Jaros{\l}aw {W{\l}odarczyk}.
\newblock {Logarithmic resolution of singularities in characteristic 0 using weighted blow-ups}.
\newblock {\em arXiv e-prints}, page arXiv:2503.13341, March 2025.

\bibitem[{Bir}12]{SingbaseFanofib}
Caucher {Birkar}.
\newblock {Singularities on the base of a Fano type fibration}.
\newblock {\em arXiv e-prints}, page arXiv:1210.2658, October 2012.

\bibitem[Bir22]{BirkarFano}
C.~Birkar.
\newblock Boundedness of {Fano} type fibrations.
\newblock {\em Ann. Sci. {\'E}cole Norm. Sup. (4)}, 2022.
\newblock To appear.

\bibitem[BM24]{bernasconi2024boundinggeometricallyintegraldel}
Fabio Bernasconi and Gebhard Martin.
\newblock {Bounding geometrically integral del Pezzo surfaces}.
\newblock {\em Forum of Mathematics, Sigma}, 12:e81, 2024.

\bibitem[BQ25a]{BirkarQu}
C.~Birkar and S.~Qu.
\newblock Stein degree on log {Calabi--Yau} fibrations, 2025.
\newblock arXiv:2509.20948.

\bibitem[BQ25b]{BirkarQuNonFano}
C.~Birkar and S.~Qu.
\newblock Stein degree on non-{Fano} type fibrations, 2025.
\newblock arXiv:2509.21824.

\bibitem[BQ25c]{sdeglogcalabiyau}
Caucher {Birkar} and Santai {Qu}.
\newblock {Stein degree on log Calabi-Yau fibrations}.
\newblock {\em arXiv e-prints}, page arXiv:2509.20948, September 2025.

\bibitem[BW17]{birkar2014existencemorifibrespaces}
Caucher Birkar and Joe Waldron.
\newblock {Existence of Mori fibre spaces for 3-folds in char $p$}.
\newblock {\em Advances in Mathematics}, 313:62--101, June 2017.

\bibitem[CP08]{Cossart2008ResolutionOS}
Vincent Cossart and Olivier Piltant.
\newblock {Resolution of singularities of threefolds in positive characteristic. I.: Reduction to local uniformization on Artin–Schreier and purely inseparable coverings}.
\newblock {\em Journal of Algebra}, 320:1051--1082, 2008.

\bibitem[CP09]{Cossart2009RESOLUTIONOS}
Vincent Cossart and Olivier Piltant.
\newblock {Resolution of singularities of threefolds in positive characteristic. II.}
\newblock {\em Journal of Algebra}, 321:1836--1976, 2009.

\bibitem[Cut04]{Cutkosky2004ResolutionOS}
Steven~Dale Cutkosky.
\newblock {Resolution of Singularities}.
\newblock 2004.

\bibitem[{Das}18]{DasWeakBAB}
Omprokash {Das}.
\newblock {On the boundedness of anti-canonical volumes of singular Fano $3$-folds in characteristic $p>5$}.
\newblock {\em arXiv e-prints}, page arXiv:1808.02102, August 2018.

\bibitem[FvGV91]{joinsandintersections}
Hubert Flenner, Leendert van Gastel, and Wolfgang Vogel.
\newblock Joins and intersections.
\newblock {\em Mathematische Annalen}, 291:691--704, 01 1991.

\bibitem[HNT17]{KYH20lcmmp}
Kenta {Hashizume}, Yusuke {Nakamura}, and Hiromu {Tanaka}.
\newblock {Minimal model program for log canonical threefolds in positive characteristic}.
\newblock {\em arXiv e-prints}, page arXiv:1711.10706, November 2017.

\bibitem[{Jia}25]{complement1}
Xintong {Jiang}.
\newblock {Boundedness of complements for fibered Fano threefolds in positive characteristic}.
\newblock {\em arXiv e-prints}, page arXiv:2506.00553, May 2025.

\bibitem[Jou83]{Bertini}
Jean-Pierre Jouanolou.
\newblock {\em Théorèmes de Bertini et applications}.
\newblock Progress in mathematics ; vol. 42. Birkhäuser, Boston, 1983.

\bibitem[{Sho}93]{3foldflip}
V.~V. {Shokurov}.
\newblock {3-FOLD Log Flips}.
\newblock {\em Izvestiya: Mathematics}, 40(1):95--202, February 1993.

\bibitem[{Sta}26]{stacks-project}
The {Stacks project authors}.
\newblock The stacks project.
\newblock \url{https://stacks.math.columbia.edu}, 2026.

\bibitem[{Tan}22]{Bertinibasechange}
Hiromu {Tanaka}.
\newblock {Bertini theorems admitting base changes}.
\newblock {\em arXiv e-prints}, page arXiv:2208.00254, July 2022.

\end{thebibliography}
\end{document}